\documentclass{article}

\usepackage[margin=0.8in]{geometry}
\usepackage{amsthm}
\usepackage{amsmath}
\usepackage{amsfonts}
\usepackage{amssymb}
\usepackage{algorithm2e}
\usepackage{graphicx}
\usepackage{subcaption}
\usepackage{color}
\usepackage{hyperref}

\title{MAGMA: Multi-level accelerated gradient mirror descent algorithm for large-scale convex composite minimization}
\author{Vahan Hovhannisyan \and Panos Parpas \and Stefanos Zafeiriou}

\newcommand{\Tau}{\mathcal{T}}

\DeclareMathOperator*{\dom}{dom}
\DeclareMathOperator*{\prox}{prox}
\DeclareMathOperator*{\argmin}{arg\,min}
\DeclareMathOperator*{\Prog}{Prog}
\DeclareMathOperator*{\Mirr}{Mirr}
\DeclareMathOperator*{\D}{D}

\newtheorem{theorem}{Theorem}
\newtheorem{lemma}{Lemma}
\newtheorem{definition}{Definition}
\newtheorem{remark}{Remark}
\newtheorem{assumption}{Assumption}

\begin{document}

\maketitle

\begin{abstract}
Composite convex optimization models arise in several applications, and are especially prevalent in inverse problems with a sparsity inducing norm and in general convex optimization with simple constraints. The most widely used algorithms for convex composite models are accelerated first order methods, however they can take a large number of iterations to compute an acceptable solution for large-scale problems. In this paper we propose to speed up first order methods by taking advantage of the structure present in many applications and in image processing in particular. Our method is based on multi-level optimization methods and exploits the fact that many applications that give rise to large scale models can be modelled using varying degrees of fidelity. We use Nesterov's acceleration techniques together with the multi-level approach to achieve an $\mathcal{O}(1/\sqrt{\epsilon})$ convergence rate, where $\epsilon$ denotes the desired accuracy. The proposed method has a better convergence rate than any other existing multi-level method for convex problems, and in addition has the same rate as accelerated methods, which is known to be optimal for first-order methods. {Moreover}, as our numerical experiments show, on large-scale face recognition problems our algorithm is several times faster than the state of the art.
\end{abstract}

\section{Introduction}

Composite convex optimization models consist of the optimization of a convex function with Lipschitz continuous gradient plus a non-smooth term. They arise in several applications, and are especially prevalent in inverse problems with a sparsity inducing norm and in general convex optimization with simple constraints (e.g. box or linear constraints). Applications include signal/image reconstruction from few linear measurements (also referred to as compressive or compressed sensing) \cite{candes2006compressive,candes2006robust,candes2008introduction,donoho2006compressed}, image super-resolution \cite{yang2010image}, data classification \cite{wright2010sparse,wright2009robust}, feature selection \cite{saeys2007review}, sparse matrix decomposition \cite{candes2006robust}, trace norm matrix completion \cite{candes2009exact}, image denoising and deblurring \cite{engl1996regularization}, to name a few.

Given the large number of applications it is not surprising that several specialized algorithms have been proposed for convex composite models. Second order methods outperform all other methods in terms of the number of iterations required. Interior point methods \cite{bertsekas1999nonlinear}, Newton and truncated Newton methods \cite{kim2007interior}, primal and dual Augmented Lagrangian methods \cite{yang2013fast} have all been proposed. However, unless the underlying application posses some specific sparsity pattern second order methods are too expensive to use in applications that arise in practice. As a result the most widely used methods are first order methods. Following the development of accelerated first order methods for the smooth case \cite{nesterov1983method} and the adaptation of accelerated first order methods to the non-smooth case (\cite{beck2009fast,nesterov2013gradient}) there has been a large amount of research in this area. State of the art methods use only first order information and as a result (even when accelerated) take a large number of iterations to compute an acceptable solution for large-scale problems. Applications in image processing can have millions of variables and therefore new methods that take advantage of problem structure are needed.

We propose to speed up first order methods by taking advantage of the structure present in many applications and in image processing in particular. The proposed methodology uses the fact that many applications that give rise to large scale models can be modelled using varying degrees of fidelity. The typical applications of convex composite optimization include dictionaries for learning problems, images for computer vision applications, or discretization of infinite dimensional problems  appearing in image processing \cite{aubert2006mathematical}. First order methods use a quadratic model with first order information and the Lipschitz constant to construct a search direction. We propose to use the solution of  a low dimensional (not necessarily quadratic) model to compute a search direction. The low dimensional model is not just a low fidelity model of the original model but it is constructed in a way so that is consistent with the high fidelity model. 

The method we propose is based on multi-level optimization methods. A multi-level method for solving convex infinite-dimensional optimization problems was introduced in \cite{braess1999multigrid} and later further developed by Nash in \cite{nash2000multigrid}. Although \cite{nash2000multigrid} is one of the pioneering works that uses a multi-level algorithm in an optimization context, it is rather abstract and only gives a general idea of how an optimization algorithm can be used in a multi-level way. The author proves that the algorithm converges under certain mild conditions. Extending the key ideas of Nash's multi-level method in \cite{wen2009line}, Wen and Goldfarb used it in combination with a line search algorithm for solving discretized versions of unconstrained infinite-dimensional smooth optimization problems. The main idea is to use the solution obtained from a coarse level for calculating the search direction in the fine level. On each iteration the algorithm uses either the negative gradient direction on the current level or a direction obtained from coarser level models. They prove that either search direction is a descent direction using Wolfe-Armijo and Goldstein-like stopping rules in their backtracking line search procedure. Later a multi-level algorithm using the trust region framework was proposed in \cite{gratton2008recursive}. In all those works the objective function is assumed to be smooth, whereas the problems addressed in this paper are not. We also note that multi-level optimization methods have been applied to PDE optimization for many years. A review of the latter literature appeared in \cite{borzi2009multigrid}.

The first multi-level optimization algorithm covering the non-smooth convex composite case was introduced in \cite{parpas2014multilevel}. It is a multi-level version of the well-known Iterative Shrinkage-Thresholding Algorithm (ISTA) (\cite{beck2009fast}, \cite{chambolle1998nonlinear}, \cite{daubechies2004iterative}), called MISTA. In \cite{parpas2014multilevel} the authors prove R-linear convergence rate of the algorithm and demonstrate in several image deblurring examples that MISTA can be up to 3-4 times faster than state of the art methods. However, MISTA's R-linear convergence rate is not optimal for first order methods and thus it has the potential for even better performance. Motivated by this possible improvement, in this paper we apply Nesterov's acceleration techniques on MISTA, and show that it is possible to achieve $\mathcal{O}(1/\sqrt{\epsilon})$ convergence rate, where $\epsilon$ denotes the desired accuracy. To the best of our knowledge this is the first multi-level method that has an optimal rate. In addition to the proof that our method is optimal, we also show in numerical experiments that, for large-scale problems it can be up to $10$ times faster than the current state of the art. 

{One very popular recent approach for handling very large scale problems are stochastic coordinate descent methods \cite{fercoq2015accelerated, lin2015accelerated}. However, while being very effective for $\ell_1$-regularized least squares problems in general, stochastic gradient methods are less applicable for problems with highly correlated data, such as the face recognition problem and other pattern recognition problems, as well as in dense error correction problems with highly correlated dictionaries \cite{wright2010dense}. We compare our method both with accelerated gradient methods (FISTA) as well as stochastic block-coordinate methods (APCG \cite{fercoq2015accelerated}).} 

Our convergence proof is based on the proof of Nesterov's method given in \cite{allen2014novel}, where the authors showed that optimal gradient methods can be viewed as a linear combination of primal gradient and mirror descent steps. We show that for some iterations (especially far away from the solution) it is beneficial to replace gradient steps with coarse correction steps. The coarse correction steps are computed by iterating on a carefully designed coarse model and a different step-size strategy. Since our algorithm combines concepts from multilevel optimization, gradient and mirror descent steps we call it Multilevel Accelerated Gradient Mirror Descent Algorithm (MAGMA). {The proposed method converges in function value with a faster rate than any other existing multi-level method for convex problems}, and in addition has the same rate as accelerated methods which is known to be optimal for first order methods \cite{nemirovski1982problem}. However, given that we use the additional structure of problems that appear in imaging applications in practice our algorithm is several times faster than the state of the art.

The rest of this paper is structured as follows: in Section \ref{sec:problem statement} we give mathematical definitions and describe the algorithms most related to MAGMA. Section \ref{sec:mult-level algorithm} is devoted to presenting a multi-level model as well as our multi-level accelerated algorithm with its convergence proof. Finally, in Section \ref{sec:experiments} we present the numerical results from experiments on large-scale face recognition problems.

\section{Problem Statement and Related Methods} \label{sec:problem statement}
In this section we give a precise description of the problem class MAGMA can be applied to and discuss the main methods related to our algorithm. 
MAGMA uses gradient and mirror descent steps, but for some iterations it uses a smooth coarse model to compute the search direction. So in order to understand the main components of the proposed algorithm we briefly review gradient descent, mirror descent and smoothing techniques in convex optimization. We also introduce our main assumptions and some notation that will be used later on.

\subsection{Convex Composite Optimization}
Let $f:\mathbb{R}^n\rightarrow (-\infty,\infty]$ be a convex, continuously differentiable function with an $L_f$-Lipschitz continuous gradient:
\begin{equation*}
\Vert \nabla f(\mathbf{x})-\nabla f(\mathbf{y})\Vert_* \leq L_f \Vert \mathbf{x}-\mathbf{y}\Vert,
\end{equation*}
where $\Vert \cdot\Vert_{*}$ is the dual norm of $\Vert \cdot\Vert$ and $g:\mathbb{R}^n\rightarrow (-\infty,\infty]$ be a closed, proper, locally Lipschitz continuous convex function, but not necessarily differentiable. Assuming that the minimizer of the following optimization problem:
\begin{equation} \label{eq:composite}
\min_{\mathbf{x}\in \mathbb{R}^n}\{F(\mathbf{x})=f(\mathbf{x})+g(\mathbf{x})\},
\end{equation}
is attained, we are concerned with solving it.

An important special case of (\ref{eq:composite}) is when {$f(\mathbf{x})=\frac{1}{2}\Vert A\mathbf{x}-\mathbf{b}\Vert_2^2$} and $g(\mathbf{x})=\lambda\Vert\mathbf{x}\Vert_1$, where $\mathbf{A}\in\mathbb{R}^{m\times n}$ is a full rank matrix  with $m<n$, $\mathbf{b}\in\mathbb{R}^m$ is a vector and $\lambda>0$ is a parameter. 
Then (\ref{eq:composite}) becomes
\begin{equation} \label{eq:L1regLS}
\tag{$P_1$}
\min\limits_{\mathbf{x}\in\mathbb{R}^n}{\frac{1}{2}\Vert \mathbf{A}\mathbf{x}-\mathbf{b}\Vert_2^2+\lambda\Vert \mathbf{x}\Vert_1},
\end{equation}
The problem (\ref{eq:L1regLS}) arises from solving underdetermined linear system of equations $\mathbf{A}\mathbf{x}=\mathbf{b}$. This system has more unknowns than equations and thus has infinitely many solutions. A common way of avoiding that obstacle and narrowing the set of solutions to a well-defined one is regularization via sparsity inducing norms \cite{elad2010sparse}.  In general, the problem of finding the sparsest decomposition of a matrix $\mathbf{A}$ with regard to data sample $\mathbf{b}$ can be written as the following non-convex minimization problem:
\begin{equation} \label{eq:l0-norm}
\tag{$P_0$}
\begin{array}{ccc}
\min\limits_{\mathbf{x}\in\mathbb{R}^n}\Vert\mathbf{x}\Vert_{0} & \text{subject to} &  \mathbf{A}\mathbf{x}=\mathbf{b},
\end{array}
\end{equation}
where $\Vert\cdot\Vert_{0}$ denotes the $\ell_0$-pseudo-norm, i.e. counts the number of non-zero elements of $\mathbf{x}$. Thus the aim is to recover the sparsest $\mathbf{x}$ such that $\mathbf{A}\mathbf{x}=\mathbf{b}$. However, solving the above non-convex problem is NP-hard, even difficult to approximate \cite{amaldi1998approximability}. Moreover, less sparse, but more stable to noise solutions are often more preferred than the sparsest one. These issues can be addressed by minimizing the more forgiving $\ell_1$-norm instead of the $\ell_0$-pseudo-norm:
\begin{equation} \label{eq:BP}
\tag{BP}
\begin{array}{ccc}
\min\limits_{\mathbf{x}\in\mathbb{R}^n}\Vert\mathbf{x}\Vert_{1} & \text{subject to} &  \mathbf{A}\mathbf{x}=\mathbf{b}.
\end{array}
\end{equation}
It can be shown that problem (\ref{eq:BP}) (often referred to as Basis Pursuit (BP)) is convex, its solutions are gathered in a bounded convex set and at least one of them has at most $m$ non-zeros \cite{elad2010sparse}. In order to handle random noise in real world applications, it is often beneficial to allow some constraint violation. Allowing a certain error $\epsilon>0$ to appear in the reconstruction we obtain the so-called Basis Pursuit Denoising (BPD) problem:
\begin{equation} \label{eq:BPD}
\tag{BPD}
\begin{array}{ccc}
\min\limits_{\mathbf{x}\in\mathbb{R}^n}\Vert\mathbf{x}\Vert_{1} & \text{subject to} &  \Vert \mathbf{A}\mathbf{x}-\mathbf{b}\Vert_2^2 \leq \epsilon,
\end{array}
\end{equation}
or using the Lagrangian dual we can equivalently rewrite it as an unconstrained, but non-smooth problem (\ref{eq:L1regLS}). Note that problem (\ref{eq:L1regLS}) is equivalent to the well known LASSO regression \cite{tibshirani1996regression}. Often BP and BPD (in both constrained and unconstrained forms) are referred to as $\ell_1$-min problems.

A relatively new, but very popular example of the BPD problem is the so-called dense error correction (DEC) \cite{wright2010dense}, which studies the problem of recovering a sparse signal $\mathbf{x}$ from highly corrupted linear measurements $\mathbf{A}$. It was shown that if $\mathbf{A}$ has highly correlated columns, with an overwhelming probability the solution of (\ref{eq:BP}) is also a solution for (\ref{eq:l0-norm}) \cite{wright2010dense}. One example of DEC is the face recognition (FR) problem, where $\mathbf{A}$ is a dictionary of facial images, with each column being one image, $\mathbf{b}$ is a new incoming image and the non-zero elements of the sparse solution $\mathbf{x}$ identify the person in the dictionary whose image $\mathbf{b}$ is. 

Throughout this section we will continue referring to the FR problem as a motivating example to explain certain details of the theory. In the DEC setting, $\mathbf{A}$ among other structural assumptions, is also assumed to have highly correlated columns; clearly, facial images are highly correlated. This assumption on one hand means that in some sense $\mathbf{A}$ contains extra information, but on the other hand, it is essential for the correct recovery. We propose to exploit the high correlation of the columns of $\mathbf{A}$ by creating coarse models that have significantly less columns and thus can be solved much faster by a multi-level algorithm.

We use the index $H$ to indicate the coarse level of a multi-level method: $\mathbf{x}_H\in\mathbb{R}^{n_H}$ is the coarse level variable, $f_H(\cdot):\mathbb{R}^{n_H}\rightarrow (-\infty,\infty]$ and $g_H(\cdot):\mathbb{R}^{n_H}\rightarrow (-\infty,\infty]$ are the corresponding coarse surrogates of $f(\cdot)$ and $g(\cdot)$, respectively. In theory, these functions only need to satisfy a few very general assumptions (Assumption \ref{ass:coarse-model}), but of course, in practice they should reflect the structure of the problem, so that the coarse model is a meaningful smaller sized representation of the original problem. In the face recognition example creating coarse models for a dictionary of facial images is fairly straightforward: we take linear combinations of the columns of $\mathbf{A}$.

In our proposed method we do not use the direct prolongation of the coarse level solution to obtain a solution for the fine level problem. Instead, we use these coarse solutions to construct search directions for the fine level iterations. Therefore, in order to ensure the convergence of the algorithm it is essential to ensure that the fine and coarse level problems are coherent in terms of their gradients \cite{nash2000multigrid,wen2009line}. However, in our setting the objective function is not smooth. To overcome this in a consistent way, we construct smooth coarse models, as well as use the smoothed version of the fine problem to find appropriate step-sizes\footnote{We give more details about smoothing non-smooth functions in general, and $\Vert\cdot\Vert_1$ in particular, at the end of this section.}. Therefore, we make the following natural assumptions on the coarse model.
\begin{assumption} \label{ass:coarse-model}
For each coarse model constructed from a fine level problem (\ref{eq:composite}) it holds that
\begin{enumerate}
\item The coarse level problem is solvable, i.e. bounded from below and the minimum is attained, and
\item $f_H(\mathbf{x})$ and $g_H(\mathbf{x})$ are both continuous, closed, convex and differentiable with Lipschitz continuous gradients.
\end{enumerate}
\end{assumption}

On the fine level, we say that iteration $k$ is a coarse step, if the search direction is obtained by coarse correction. To denote the $j$-th iteration of the coarse model, we use $\mathbf{x}_{H,j}$.

\subsection{Gradient Descent Methods}
Numerous methods have been proposed to solve (\ref{eq:composite}) when $g$ has a simple structure. By simple structure we mean that its proximal mapping, defined as
\begin{equation*}
\prox\nolimits_{L}(\mathbf{x})=\argmin_{{\mathbf{y}}\in\mathbb{R}^n} \bigg\lbrace \frac{L}{2}\Vert{\mathbf{y}}-\mathbf{x}\Vert^2 + \langle\nabla f(\mathbf{x}),\mathbf{y}-\mathbf{x} \rangle + g(\mathbf{y})\bigg\rbrace,
\end{equation*}
for some $L\in\mathbb{R}$, is given in a closed {form for any $f$}. Correspondingly, denote
\begin{equation*}
\Prog\nolimits_L(\mathbf{x})=-\min_{{\mathbf{y}}\in\mathbb{R}^n} \bigg\lbrace \frac{L}{2}\Vert{\mathbf{y}}-\mathbf{x}\Vert^2 + \langle\nabla f(\mathbf{x}),\mathbf{y}-\mathbf{x} \rangle + g(\mathbf{y}) - g(\mathbf{x}) \bigg\rbrace.
\end{equation*}
We often use the $\prox$ and $\Prog$ operators with the Lipschitz constant $L_f$ of $\nabla f$. In that case we simplify the notation to $\prox(\mathbf{x})=\prox\nolimits_{L_f}(\mathbf{x})$ and $\Prog(\mathbf{x})=\Prog\nolimits_{L_f}(\mathbf{x})$ \footnote{These definitions without loss of generality also cover the case when the problem is constrained and smooth. Also see \cite{nesterov2005smooth} for the original definitions for the unconstrained smooth case.}.

In case of the FR problem (and LASSO, in general) {$\text{prox}(\mathbf{x}-\frac{1}{L}\nabla f(\mathbf{x}))$} becomes the shrinkage-thresholding operator $\Tau_L(\mathbf{x})_i=(\vert \mathbf{x}_i\vert-1/L)_{+} \text{sgn}(\mathbf{x}_i)$ and the iterative scheme is given by,
\begin{equation*}
\mathbf{x}_{k+1}=\Tau_{\lambda/L}\left(\mathbf{x}_k -\frac{2}{L} \mathbf{A}^\top( \mathbf{A}\mathbf{x}_k-\mathbf{b})\right).
\end{equation*}
Then problem (\ref{eq:L1regLS}) can be solved by the well-known Iterative Shrinkage Thresholding Algorithm (ISTA) and its generalizations, see, e.g. \cite{beck2009fast}, \cite{chambolle1998nonlinear}, \cite{daubechies2004iterative}, \cite{figueiredo2003algorithm}. When the stepsize is fixed to the Lipschitz constant of the problem, ISTA reduces to the following simple iterative scheme,
\begin{equation*}
\mathbf{x}_{k+1}=\prox(\mathbf{x}_{k}).
\end{equation*}
If $g(\cdot)\equiv 0$, the problem is smooth, $\prox({\mathbf{x}})\equiv \mathbf{x}_k-\frac{1}{L_f}\nabla f(\mathbf{x}_k)$ and ISTA becomes the well-known steepest descent algorithm with
\begin{equation*}
\mathbf{x}_{k+1}=\mathbf{x}_k-\frac{1}{L_f}\nabla f(\mathbf{x}_k).
\end{equation*}
For composite optimization problems it is common to use the gradient mapping as an optimality measure \cite{nesterov2004introductory}:
\begin{equation*}
\D(\mathbf{x})=\mathbf{x}-\prox(\mathbf{x}).
\end{equation*}
It is a generalization of the gradient for the composite setting and preserves its most important properties used for the convergence analysis.

We will make use of the following fundamental Lemma in our convergence analysis. The proof can be found, for instance, in \cite{nesterov2004introductory}.
\begin{lemma}[Gradient Descent Guarantee] \label{l:gradient-descent-guarantee}
For any $\mathbf{x}\in\mathbb{R}^n$,
\begin{equation*}
F(\prox(\mathbf{x}))\leq F(\mathbf{x})-\Prog(\mathbf{x}).
\end{equation*}
\end{lemma}

Using Lemma \ref{l:gradient-descent-guarantee}, it can be shown that ISTA computes an $\epsilon$-optimal solution in terms of function values in $\mathcal{O}(1/\epsilon)$ iterations (see \cite{beck2009fast} and references therein).

\subsection{Mirror Descent Methods}
Mirror Descent methods solve the general problem of minimizing a proper convex and locally Lipschitz continuous function $F$ over a simple convex set $Q$ by constructing lower bounds to the objective function (see for instance \cite{nemirovski1982problem}, \cite{nesterov2009primal}). Central to the definition of mirror descent is the concept of a distance generating function.
\begin{definition}
A function $\omega(\mathbf{x}):Q\rightarrow\mathbb{R}$ is called a Distance Generating Function {(DFG)}, if $\omega$ is $1$-strongly convex with respect to a norm $\Vert\cdot\Vert$, that is
\begin{equation*}
\omega(\mathbf{z})\geq\omega(\mathbf{x})+\langle\nabla\omega(\mathbf{x}),\mathbf{z}-\mathbf{x}\rangle+\frac{1}{2}\Vert \mathbf{x}-\mathbf{z}\Vert^2 \quad\forall \mathbf{x},\mathbf{z}\in Q.
\end{equation*}
Accordingly, the Bregman divergence (or prox-term) is given as
\begin{equation*}
V_\mathbf{x}(\mathbf{z})=\omega(\mathbf{z})-\langle\nabla\omega(\mathbf{x}),\mathbf{z}-\mathbf{x}\rangle-\omega(\mathbf{x}) \quad\forall \mathbf{x},\mathbf{z}\in Q.
\end{equation*}
The property of DFG ensures that $V_\mathbf{x}(\mathbf{x})=0$ and $V_\mathbf{x}(\mathbf{z})\geq\frac{1}{2}\Vert \mathbf{x}-\mathbf{z}\Vert^2\geq 0$.
\end{definition}

The main step of mirror descent algorithm is given by,
\begin{equation*}
\mathbf{x}_{k+1} = \Mirr\nolimits_{\mathbf{x}_k}(\nabla f(\mathbf{x}_k),\alpha), \quad \text{where} \quad \Mirr\nolimits_\mathbf{x}(\mathbf{\xi},\alpha)=\argmin_{\mathbf{z}\in \mathbb{R}^n}\{V_\mathbf{x}(\mathbf{z})+\langle\alpha\mathbf{\xi},\mathbf{z}-\mathbf{x}\rangle + \alpha g(\mathbf{z})\}.
\end{equation*}
Here again, we assume that $\Mirr$ can be evaluated in closed form. Note that choosing $\omega(\mathbf{z})=1/2\Vert \mathbf{z}\Vert_2^2$ as the DFG, which is strongly convex w.r.t. to the $\ell_2$-norm over any convex set and correspondingly $V_\mathbf{x}(\mathbf{z})=\frac{1}{2}\Vert \mathbf{x}-\mathbf{z}\Vert_2^2$, the mirror step becomes exactly the proximal step of ISTA. However, regardless of this similarity, the gradient and mirror descent algorithms are conceptually different and use different techniques to prove convergence. The core lemma for mirror descent convergence is the so called Mirror Descent Guarantee. For the unconstrained\footnote{{Simple constraints can be included by defining $g$ as an indicator function.}} composite setting it is given below, a proof can be found in \cite{nesterov2004introductory}.
\begin{lemma}[Mirror Descent Guarantee] \label{l:mirror-descent-guarantee-composite}
Let $\mathbf{x}_{k+1}=\Mirr\nolimits_{\mathbf{x}_k}(\nabla f(\mathbf{x}_k), \alpha)$, then $\forall \mathbf{u}\in \mathbb{R}^n$
\begin{equation*}
\begin{split}
\alpha(F(\mathbf{x}_k)-F(\mathbf{u})) & \leq \alpha\langle \nabla f(\mathbf{x}_k),\mathbf{x}_k-\mathbf{u}\rangle + \alpha (g(\mathbf{x}_k)-g(\mathbf{u})) \\
 & \leq \alpha^2 L_f\Prog(\mathbf{x}_k) + V_{\mathbf{x}_k}(\mathbf{u})-V_{\mathbf{x}_{k+1}}(\mathbf{u}).
\end{split}
\end{equation*}
\end{lemma}
Using the Mirror Descent Guarantee, it can be shown that the Mirror Descent algorithm converges to the minimizer $\mathbf{x}^*$ in {$\mathcal{O}(L_F/\epsilon)$}, where $\epsilon$ is the convergence precision \cite{bental2013lectures}.

\subsection{Accelerated First Order Methods}
A faster method for minimizing smooth convex functions with asymptotically tight $\mathcal{O}(1/\sqrt{\epsilon})$ convergence rate \cite{nemirovski1982problem} was proposed by Nesterov in \cite{nesterov1983method}. This method and its variants were later extended to solve the more general problem (\ref{eq:composite}) (e.g. \cite{beck2009fast}, \cite{nesterov2013gradient}, \cite{allen2014novel}), see \cite{tseng2008accelerated} for full review and analysis of accelerated first order methods. The main idea behind the accelerated methods is to update the iterate based on a linear combination of previous iterates rather than only using the current one as in gradient or mirror descent algorithms. The first and most popular method that achieved an optimal convergence rate for composite problems is FISTA \cite{beck2009fast}. At each iteration it performs one proximal operation, then uses a linear combination of its result and the previous iterator for the next iteration.

\begin{algorithm}[H] \label{alg:FISTA}
	\SetAlgoLined
	\KwData{FISTA($f(\cdot)$, $g(\cdot)$, $\mathbf{x}_{0}$)}
	Choose $\epsilon>0$.\\
	Set $\mathbf{y}_1=\mathbf{x}_0$ and $t_1=1$.\\
	\For{$k=0,1,2,...$}{
		Set $	\mathbf{x}_{k+1}=\prox(\mathbf{x}_k)$. \\
		Set $t_{k+1}=\frac{1+\sqrt{1+4t_k^2}}{2}$. \\
		Set $\mathbf{y}_{k+1}=\mathbf{x}_k+\frac{t_k-1}{t_{k+1}}\left( \mathbf{x}_k-\mathbf{x}_{k-1}\right)$. \\
		\If{$\Vert D(\mathbf{x}_k)\Vert_*<\epsilon$}{\Return $\mathbf{x}_k$.}
 	}
	\caption{FISTA}
\end{algorithm}

Alternatively, Nesterov's accelerated scheme can be modified to use two proximal operations at each iteration and linearly combine their results as the current iteration point \cite{nesterov2005smooth}. In a recent work this approach was reinterpreted as a combination of gradient descent and mirror descent algorithms \cite{allen2014novel}. The algorithm is given next.

\begin{algorithm}[H] \label{alg:AGM}
	\SetAlgoLined
	\KwData{AGM($f(\cdot)$, $g(\cdot)$, $\mathbf{x}_{0}$)}
	Choose $\epsilon>0$.\\
	Choose DGF $\omega$ and set $V_\mathbf{x}(\mathbf{y})=\omega(\mathbf{y})-\langle \nabla\omega(\mathbf{x}),\mathbf{y}-\mathbf{x}\rangle-\omega(\mathbf{x}).$\\
	Set $\mathbf{y}_0=\mathbf{z}_0=\mathbf{x}_0$.\\
	\For{$k=0,1,2,...$}{
		Set $\alpha_{k+1}=\frac{k+2}{2L}$ and $ t_k=\frac{2}{k+2}$.\\
		Set $\mathbf{x}_k= t_k \mathbf{z}_k+(1- t_k)\mathbf{y}_k$. \\
		\If{$\Vert D(\mathbf{x}_k)\Vert_*<\epsilon$}{\Return $\mathbf{x}_k$.}
		Set $\mathbf{y}_{k+1}=\prox(\mathbf{x}_k)$. \\
		Set $\mathbf{z}_{k+1}=\Mirr\nolimits_{\mathbf{z}_k}(\nabla f(\mathbf{x}_k),\alpha_{k+1})$. \\
 	}
	\caption{AGM}
\end{algorithm}

The convergence of Algorithm \ref{alg:AGM} relies on the following lemmas.
\begin{lemma} \label{l:AGM-main}
If $ t_k=\frac{1}{\alpha_{k+1} L_f}$ then it satisfies that for every $\mathbf{u}\in\mathbb{R}^n$,
\begin{equation*}
\begin{split}
& \alpha_{k+1}\langle \nabla f(\mathbf{x}_k),\mathbf{z}_k-\mathbf{u}\rangle +\alpha_{k+1}(g(\mathbf{z}_k)-g(\mathbf{u})) \\ 
& \leq \alpha_{k+1}^2 L_f\Prog(\mathbf{x}_k)+V_{\mathbf{z}_k}(\mathbf{u})-V_{\mathbf{z}_{k+1}}(\mathbf{u}) \\
& \leq \alpha_{k+1}^2 L_f(F(\mathbf{x}_k)-F(\mathbf{y}_{k+1}))+V_{\mathbf{z}_k}(\mathbf{u})-V_{\mathbf{z}_{k+1}}(\mathbf{u}).
\end{split}
\end{equation*}
\end{lemma}

\begin{lemma}[Coupling] \label{l:AGM-coupling}
For any $\mathbf{u}\in\mathbb{R}^n$
\begin{equation*}
\alpha_{k+1}^2 L_f F(\mathbf{y}_{k+1}) - (\alpha_{k+1}^2 L_f-\alpha_{k+1})F(\mathbf{y}_k) + (V_{\mathbf{z}_{k+1}}(\mathbf{u})-V_{\mathbf{z}_k}(\mathbf{u})) \leq \alpha_{k+1} F(\mathbf{u}).
\end{equation*}
\end{lemma}

\begin{theorem}[Convergence]
Algorithm \ref{alg:AGM} ensures that
\begin{equation*}
F(\mathbf{y}_T)-F(\mathbf{x}^{*}) \leq \frac{4\Theta L_f}{T^2},
\end{equation*}
where $\Theta$ is any upper bound on $V_{\mathbf{x}_0}(\mathbf{x}^{*})$.
\end{theorem}

\begin{remark}
In \cite{allen2014novel} the analysis of the algorithm is done for smooth constrained problems, however it can be set to work for the unconstrained composite setting as well.
\end{remark}

\begin{remark}
Whenever $L_f$ is not known or is expensive to calculate, we can use backtracking line search with the following stopping condition
\begin{equation} \label{eq:backtracking}
F(\mathbf{y}_{k+1})\leq F(\mathbf{x}_k)-\Prog\nolimits_L(\mathbf{x}_k),
\end{equation}
where $L>0$ is the smallest number for which the condition (\ref{eq:backtracking}) holds. The convergence analysis of the algorithm can be extended to cover this case in a straightforward way (see \cite{beck2009fast}).
\end{remark}

\subsection{Smoothing and First Order Methods}
A more general approach of minimizing non-smooth problems is to replace the original problem by a sequence of smooth problems. The smooth problems can be solved more efficiently than by using subgradient type methods on the original problem. {In \cite{nesterov2005smooth} Nesterov} proposed a first order smoothing method, where the objective function is of special $\max$-type. Then Beck and Teboulle extended this method to a more general framework for the class of so-called \textit{smoothable} functions \cite{beck2012smoothing}. Both methods are proven to converge to an $\epsilon$-optimal solution in $\mathcal{O}(1/\epsilon)$ iterations.

\begin{definition}
Let $g:\mathbb{E}\rightarrow(-\infty,\infty]$ be a closed and proper convex function and let $\mathbf{x}\subseteq\dom g$ be a closed convex set. The function $g$ is called "$(\alpha,\beta,K)$-smoothable" over $X$, if there exist $\beta_1$, $\beta_2$ satisfying $\beta_1+\beta_2=\beta>0$ such that for every $\mu>0$ there exists a continuously differentiable convex function $g_\mu :\mathbb{E}\rightarrow (-\infty,\infty)$ such that the following hold:
\begin{enumerate}
\item $g(\mathbf{x})-\beta_1\mu\leq g_\mu (\mathbf{x})\leq g(\mathbf{x})+\beta_2\mu$ for every $\mathbf{x}\in X$.
\item The function $g_\mu$ has a Lipschitz gradient over $X$ with a Lipschitz constant such that there exists $K>0$, $\alpha>0$ such that
\begin{equation*}
\Vert \nabla g_\mu(\mathbf{x})-\nabla g_\mu({\mathbf{y}})\Vert_{*}\leq \left( K+\frac{\alpha}{\mu}\right)\Vert \mathbf{x}-{\mathbf{y}}\Vert\; for\; every\; \mathbf{x},{\mathbf{y}}\in X.
\end{equation*}
\end{enumerate}
We often use $L_\mu=K+\frac{\alpha}{\mu}$.
\end{definition}

It was shown in \cite{beck2012smoothing} that with an appropriate choice of the smoothing parameter $\mu$ a solution obtained by smoothing the original non-smooth problem and solving it by a method with $\mathcal{O}(1/\sqrt{\epsilon})$ convergence rate finds an $\epsilon$-optimal solution in $\mathcal{O}(1/\epsilon)$ iterations. When the proximal step computation is tractable, accelerated first order methods are superior both in theory and in practice. However, for the purposes of establishing a step-size strategy for MAGMA it is convenient to work with smooth problems. We combine the theoretical superiority of non-smooth methods with the rich mathematical structure of smooth models to derive a step size strategy that guarantees convergence. MAGMA does not use smoothing in order to solve the original problem. Instead, it uses a smoothed version of the original model to compute a step size when the search direction is given by the coarse model (coarse correction step).

For the FR problem (and (\ref{eq:L1regLS}) in general), where $g(\mathbf{x})=\lambda \Vert \mathbf{x}\Vert_1$, it can be shown {\cite{beck2012smoothing}} that 
\begin{equation} \label{eq:g_mu}
g_\mu(\mathbf{x})=\lambda\sum_{j=1}^{n}\sqrt{\mu^2+\mathbf{x}_{j}^2}
\end{equation}
is a $\mu$-smooth approximation of $g(\mathbf{x})=\lambda\Vert \mathbf{x}\Vert_1$ with parameters $(\lambda, \lambda n, 0)$.

\section{Multi-level Accelerated Proximal Algorithm} \label{sec:mult-level algorithm}

In this section we formally present our proposed algorithm within the multi-level optimization setting, together with the proof of its convergence with $\mathcal{O}(1/\sqrt{\epsilon})$ rate, where $\epsilon$ is the convergence precision.

\subsection{Model Construction}
First we present the notation and construct a mathematical model, which will later be used in our algorithm for solving (\ref{eq:composite}). A well-defined and converging multi-level algorithm requires appropriate information transfer mechanisms in between levels and a well-defined and coherent coarse model. These aspects of our model are presented next.

\subsubsection{Information Transfer}
In order to transfer information  between levels, we use linear restriction and prolongation operators $\mathbf{R}:\mathbb{R}^{n\times n_H}$ and $\mathbf{P}:\mathbb{R}^{n_H\times n}$ respectively, where $n_H<n$ is the size of the coarse level variable. The restriction operator $\mathbf{R}$ transfers the current fine level point to the coarse level and the prolongation operator $\mathbf{P}$ constructs a search direction for the fine level from the coarse level solution. The techniques we use are standard (see \cite{briggs2000multigrid}) so we keep this section brief.

There are many ways to design $\mathbf{R}$ and $\mathbf{P}$ operators, but they should satisfy the following condition to guarantee the convergence of the algorithm,
\begin{equation*}
\sigma \mathbf{P} = \mathbf{R}^\top,
\end{equation*}
with some scalar $\sigma>0$. Without loss of generality we assume $\sigma=1$, which is a standard assumption in the multi-level literature \cite{briggs2000multigrid}, \cite{nash2000multigrid}.

In our applications, when the problem is given as (\ref{eq:L1regLS}), we use a simple interpolation operator given in the form of (\ref{eq:Rx}) as the restriction operator, which essentially constructs linear combinations of the columns of dictionary $\mathbf{A}$. We give more details on using $\mathbf{R}$ for our particular problem in Section \ref{sec:experiments}.

\subsubsection{Coarse Model} \label{sec:coarse-model}
A key property of the coarse model is that at its initial point $\mathbf{x}_{H,0}$ the optimality conditions of the two models match. This is achieved by adding a linear term to the coarse objective function:
\begin{equation} \label{eq:coarse-objective}
F_H(\mathbf{x}_H)=f_H(\mathbf{x}_H)+g_H(\mathbf{x}_H)+\langle \mathbf{v}_H,\mathbf{x}_H\rangle,
\end{equation}
where the vector $\mathbf{v}_H\in\mathbb{R}^{n_H}$ is defined so that the fine and coarse level problems are first order coherent, that is, their first order optimality conditions coincide. In this paper we have a non-smooth objective function and assume the coarse model is smooth \footnote{This can be done by smoothing the non-smooth coarse term.} with $L_H$-Lipschitz continuous gradient. Furthermore, we construct $\mathbf{v}_H$ such that the gradient of the coarse model is equal to the gradient of the smoothed fine model's gradient:
\begin{equation} \label{eq:coherence}
\nabla F_H(\mathbf{R}\mathbf{x})=\mathbf{R}\nabla F_\mu(\mathbf{x}),
\end{equation}
where $F_\mu(\mathbf{x})$ is a $\mu$-smooth approximation of $F(\mathbf{x})$ with parameters $(a,\beta,K)$. Note that for the composite problem (\ref{eq:composite}) $F_\mu(\mathbf{x})=f(\mathbf{x}) + g_\mu(\mathbf{x})$, where $g_\mu$ is a $\mu$-smooth approximation of $g(\mathbf{x})$. In our experiments for (\ref{eq:L1regLS}) we use $g_\mu(\mathbf{x})$ as given in (\ref{eq:g_mu}). The next lemma gives a choice for $\mathbf{v}_H$, such that (\ref{eq:coherence}) is satisfied.

\begin{lemma}[Lemma 3.1 of \cite{wen2009line}]\label{l:coherence}
Let $F_H$ be a Lipschitz continuous function with $L_H$ Lipschitz constant, then for
\begin{equation} \label{eq:v-term}
\mathbf{v}_H=\mathbf{R} \nabla F_\mu(\mathbf{x})-(\nabla f_H(\mathbf{R}\mathbf{x})+\nabla g_H(\mathbf{R}\mathbf{x}))
\end{equation}
we have
\begin{equation*}
\nabla F_H(\mathbf{R}\mathbf{x})=\mathbf{R} \nabla F_\mu(\mathbf{x}).
\end{equation*}
\end{lemma}
The condition (\ref{eq:coherence}) is referred to as \textit{first order coherence}. It ensures that if $\mathbf{x}$ is optimal for the smoothed fine level problem, then $\mathbf{R}\mathbf{x}$ is also optimal in the coarse level.

{
While in practice it is often beneficial to use more than one levels, for the theoretical analysis of our algorithm without loss of generality we assume that there is only one coarse level. Indeed, assume that $\mathbf{I}_h^H\in\mathbb{R}^{H\times h}$ is a linear operator that transfers $\mathbf{x}$ from $\mathbf{R}^h$ to $\mathbf{R}^H$. Now, we can define our restriction operator as $\mathbf{R}=\prod_{i=1}^{l} \mathbf{I}_{h_i}^{H_i}$, where $h_1=n$, $h_{i+1}=H_{i}$ and $H_{l}=n_H$. Accordingly, the prolongation operator will be $\mathbf{P}=\prod_{i=l}^{1} \mathbf{I}_{H_i}^{h_i}$, where $\sigma_i \mathbf{I}_{H_i}^{h_i} = (\mathbf{I}_{h_i}^{H_i})^\top$. Note that this construction satisfies the assumption that $\sigma \mathbf{P}=\mathbf{R}^\top$, with $\sigma=\prod_{i=1}^{l}\sigma_i$.}

\subsection{MAGMA}
In this subsection we describe our proposed multi-level accelerated proximal algorithm for solving (\ref{eq:composite}). We call it MAGMA for Multi-level Accelerated Gradient Mirror descent Algorithm. As in \cite{nesterov2005smooth} and \cite{allen2014novel}, at each iteration MAGMA performs both gradient and mirror descent steps, then uses a convex combination of their results to compute the next iterate. The crucial difference here is that whenever a coarse iteration is performed, our algorithm uses the coarse direction
\begin{equation*}
\mathbf{y}_{k+1}=\mathbf{x}_{k} + s_{k} \mathbf{d}_k(\mathbf{x}_{k}),
\end{equation*}
instead of the gradient, where $\mathbf{d}_k(\cdot)$ is the search direction and $s_k$ is an appropriately chosen step-size. Next we describe how each of these terms is constructed.

At each iteration $k$ the algorithm first decides whether to use the gradient or the coarse model to construct a search direction for the gradient descent step. This decision is based on the optimality condition at the current iterate $\mathbf{x}_{k}$: we do not want to use the coarse direction when it is not helpful, i.e. when 
\begin{itemize}
\item the first order optimality conditions are almost satisfied, or
\item the current point $\mathbf{x}_k$ is very close to the point $\tilde{\mathbf{x}}$, where a coarse iteration was last performed, as long as the algorithm has not performed too many gradient correction steps. 
\end{itemize} 
More formally, we choose the coarse direction, whenever both of the following conditions are satisfied:
\begin{equation} \label{eq:coarse-conditions}
\begin{aligned}
\Vert \mathbf{R} \nabla F_{\mu_{k}}(\mathbf{x}_k)\Vert_* > \kappa \Vert \nabla F_{\mu_{k}}(\mathbf{x}_k)\Vert_*, \\
\Vert \mathbf{x}_k-\tilde{\mathbf{x}}\Vert > \vartheta \Vert \tilde{\mathbf{x}}\Vert\quad\text{or}\quad q<K_d,
\end{aligned}
\end{equation}
where $\kappa\in(0,\min{(1,\min\Vert \mathbf{R}\Vert)})$, $\vartheta>0$ and $K_d$ are predefined constants, and $q$ is the number of consecutive gradient correction steps \cite{wen2009line}, \cite{parpas2014multilevel}.

If a coarse direction is chosen, a coarse model is constructed as described in (\ref{eq:coarse-objective}). In order to satisfy the coherence property (\ref{eq:coherence}), we start the coarse iterations with $\mathbf{x}_{H,0}=\mathbf{R} \mathbf{x}_k$, then solve the coarse model by a first order method and construct a search direction:
\begin{equation} \label{eq:coarse-direction}
\mathbf{d}_k(\mathbf{x}_k)=\mathbf{P}(\mathbf{x}_{H,N_H} - \mathbf{R}\mathbf{x}_k),
\end{equation}
where $\mathbf{x}_{H,N_H}\in\mathbb{R}^{n_H}$ is an approximate solution of the coarse model after $N_H$ iterations. Note that in practice we do not find the exact solution, but rather run the algorithm for $N_H$ iterations, until we achieve a certain acceptable $\mathbf{x}_{H,N_H}$. In our experiments we used the monotone version of FISTA \cite{beck2009fastMFISTA}, however in theory any monotone algorithm will ensure convergence. 

We next show that any coarse direction defined in (\ref{eq:coarse-direction}) is a descent direction for the smoothed fine problem.

\begin{lemma}[Descent Direction] \label{l:descent-direction}
If at iteration $k$ a coarse step is performed and suppose that $F_H(\mathbf{x}_{H,N_H})< F_H(\mathbf{x}_{H,1})$, then for any $\mathbf{x}_{k}$ it holds that
\begin{equation*}
\langle \mathbf{d}_k(\mathbf{x}_k), \nabla F_{\mu_{k}}(\mathbf{x}_k)\rangle < -\frac{\kappa^2}{2L_H} \Vert \nabla F_{\mu_{k}}(\mathbf{x}_k)\Vert_*^2 \leq 0.
\end{equation*}
\end{lemma}
\begin{proof}
Using the definition of coarse direction $\mathbf{d}(\mathbf{x}_k)$, linearity of $\mathbf{R}$ and Lemma \ref{l:coherence} we obtain
\begin{equation} \label{eq:descent-direction-1}
\begin{split}
\langle \mathbf{d}_k(\mathbf{x}_k), \nabla F_{\mu_{k}}(\mathbf{x}_k)\rangle 
& = \langle \mathbf{R}^\top(\mathbf{x}_{H,N_H} - \mathbf{x}_{H,0}), \nabla F_{\mu_{k}}(\mathbf{x}_{k}) \rangle\\
& = \langle \mathbf{x}_{H,N_H} - \mathbf{x}_{H,0}, \mathbf{R} \nabla F_{\mu_{k}}(\mathbf{x}_{k}) \rangle \\
& = \langle \mathbf{x}_{H,N_H} - \mathbf{x}_{H,0}, \nabla F_H(\mathbf{x}_{H,0}) \rangle\\
& \leq F_{H}(\mathbf{x}_{H,N_H}) - F_H(\mathbf{x}_{H,0}),
\end{split}
\end{equation}
where for the last inequality we used the convexity of $F_H$. On the other hand using the monotonicity assumption on the coarse level algorithm and lemma (\ref{l:gradient-descent-guarantee}) we derive:
\begin{equation} \label{eq:descent-direction-2}
F_{H}(\mathbf{x}_{H,N_H}) - F_H(\mathbf{x}_{H,0})
 < F_H(\mathbf{x}_{H,1}) - F_H(\mathbf{x}_{H,0})
 \leq -\frac{1}{2L_H} \Vert\nabla F_H(\mathbf{x}_{H,0})\Vert_{{*}}^2.
\end{equation}

Now using the choice $\mathbf{x}_{H,0}=\mathbf{R}\mathbf{x}_k$ and (\ref{eq:coherence}) we obtain:
\begin{equation} \label{eq:descent-direction-3}
\nabla F_H(\mathbf{x}_{H,0})
 = \nabla F_H(\mathbf{R} \mathbf{x}_k)
 = \mathbf{R}\nabla F_{\mu_{k}}(\mathbf{x}_k).
\end{equation}

Then combining (\ref{eq:descent-direction-2}), (\ref{eq:descent-direction-3}) and condition (\ref{eq:coarse-conditions}) we obtain 
\begin{equation} \label{eq:descent-direction-4}
F_{H}(\mathbf{x}_{H,N_H}) - F_H(\mathbf{x}_{H,0})
 < -\frac{1}{2L_H} \Vert\nabla \mathbf{R} F_{\mu_k}(\mathbf{x}_{k})\Vert^2
 \leq -\frac{\kappa^2}{2L_H}\Vert \nabla F_{\mu_k}(\mathbf{x}_k)\Vert^2.
\end{equation}

Finally, combining (\ref{eq:descent-direction-1}) and (\ref{eq:descent-direction-4}) we derive the desired result.
\end{proof}

After constructing a descent search direction, we find suitable step sizes for both gradient and mirror steps at each iteration. In \cite{allen2014novel} a fixed step size of $\alpha_{k+1}=\frac{1}{L_f}$ is used for gradient steps. However, in our algorithm we do not always use the steepest descent direction for gradient steps, and therefore another step size strategy is required. We show that step sizes obtained from Armijo-type backtracking line search on the smoothed (fine) problem ensures that the algorithm converges with optimal rate. Starting from a predefined {$s_k=s^0$}, we reduce it by a factor of constant $\tau\in(0,1)$ until
\begin{equation} \label{eq:line-search}
F_{\mu_{k}}(\mathbf{y}_{k+1})\leq F_{\mu_{k}}(\mathbf{x}_k) + c s_k \langle \mathbf{d}_k(\mathbf{x}_k),  \nabla F_{\mu_{k}}(\mathbf{x}_k)\rangle
\end{equation}
is satisfied for some predefined constant $c\in (0,1)$.

For mirror steps we always use the current gradient as search direction, as they have significant contribution to the convergence only when the gradient is small enough, meaning we are close to the minimizer, and in those cases we do not perform coarse iterations. However, we need to adjust the step size $\alpha_k$ for mirror steps in order to ensure convergence. Next we formally present the proposed algorithm.

\begin{algorithm}[H] \label{alg:MAGMA}
	\SetAlgoLined
	\KwData{MAGMA($f(\cdot)$, $g(\cdot)$, $\mathbf{x}_{0}$, $T$)}
	Set $\mathbf{y}_0=\mathbf{z}_0=\mathbf{x}_0$ and $\alpha_0=0$\\
	\For{$k=0,1,2,...,T$}{
		{Choose $\eta_{k+1}$ and $\alpha_{k+1}$ according to (\ref{eq:eta}) and (\ref{eq:alpha})}. \\
		Set $t_k=\frac{1}{\alpha_{k+1}\eta_{k+1}}$. \\
		Set $\mathbf{x}_{k}= t_k \mathbf{z}_{k}+(1- t_k)\mathbf{y}_{k}$. \\
		\If{Conditions (\ref{eq:coarse-conditions}) are satisfied}{
			Perform $N_H$ coarse iterations. \\
			Set $\mathbf{d}_k=\mathbf{P}(\mathbf{x}_{H,N_H} - \mathbf{R}\mathbf{x}_k)$. \\
			Set $\mathbf{y}_{k+1}=\mathbf{x}_k + s_k \mathbf{d}(\mathbf{x}_k)$, with $s_k$ satisfying (\ref{eq:line-search}). \\
		}\Else{
			Set $\mathbf{y}_{k+1}=\prox(\mathbf{x}_k)$. \\
		}
		Set $\mathbf{z}_{k+1}=\Mirr\nolimits_{\mathbf{z}_k}(\nabla f(\mathbf{x}_k),\alpha_{k+1})$. \\
 	}
	\caption{MAGMA}
\end{algorithm}

Now we show the convergence of Algorithm \ref{alg:MAGMA} with $\mathcal{O}(1/ \sqrt{\epsilon})$ rate. First we prove an analogue of Lemma \ref{l:AGM-main}.

\begin{lemma} \label{l:magma-main-lemma}
For every $\mathbf{u}\in\mathbb{R}^n$ and $\eta_k>0$ and

{\begin{equation*}
\eta_{k+1}=\left\{
    \begin{array}{ll}
      L_f,\quad\text{when}\; k+1\;\text{is a gradient correction step}\\
      \max\Big\{\frac{1}{4\alpha_k^2 \eta_k},\frac{L_H}{c s_k\kappa^2}
      \Big\},\quad\text{otherwise}\\
    \end{array}
  \right.
\end{equation*}}

{it holds that}

\begin{equation*}
\begin{split}
 & \alpha_{k+1} \langle \nabla f(\mathbf{x}_k),\mathbf{z}_k -\mathbf{u}\rangle + \alpha_{k+1}(g(\mathbf{z}_k)-g(\mathbf{u})) \\
\leq & \alpha_{k+1}^2 L_f \Prog(\mathbf{x}_{k}) + V_{{\mathbf{z}}_{k}}({\mathbf{u}})-V_{{\mathbf{z}}_{k+1}}({\mathbf{u}}) \\
\leq & \alpha_{k+1}^2 [\eta_{k+1}({F_{\mu_k}}(\mathbf{x}_k) - {F_{\mu_k}}(\mathbf{y}_{k+1})) + {L_f}\beta{\mu_{k}}] + V_{\mathbf{z}_k}(\mathbf{u})-V_{\mathbf{z}_{k+1}}(\mathbf{u})
\end{split}
\end{equation*}
\end{lemma}

\begin{proof}
First note that the {proof of the first inequality} follows directly from Lemma 4.2 of \cite{allen2014novel}. Now assume $k$ is a coarse step, then the first inequality follows from Lemma \ref{l:mirror-descent-guarantee-composite}. To show the second one we first note that for $\tilde{\mathbf{x}}=\prox(\mathbf{x}_k)$,
\begin{equation*}
\begin{split}
 & \Prog(\mathbf{x}_k)-\frac{1}{2L_f}\Vert \nabla F_{\mu_{k}}(\mathbf{x}_k)\Vert^2 \\
 = & -\frac{L_f}{2}\Vert \tilde{\mathbf{x}}-\mathbf{x}_k\Vert^2 - \langle \nabla f(\mathbf{x}_k),\tilde{\mathbf{x}}-\mathbf{x}_k\rangle - g(\tilde{\mathbf{x}})+g(\mathbf{x}_k) - \frac{1}{2L_f}\Vert \nabla F_{\mu_{k}}(\mathbf{x}_k)\Vert^2 \\
 \leq & -\frac{L_f}{2}\Vert \tilde{\mathbf{x}}-\mathbf{x}_k\Vert^2 - \langle \nabla f(\mathbf{x}_k),\tilde{\mathbf{x}}-\mathbf{x}_k\rangle - \langle \tilde{\mathbf{x}}-\mathbf{x}_k,\nabla g_{\mu_{k}} (\mathbf{x}_k)\rangle + \beta{\mu_{k}} - \frac{1}{2L_f}\Vert \nabla F_{\mu_{k}}(\mathbf{x}_k)\Vert^2 \\
 = & -\frac{L_f}{2} (\Vert \tilde{\mathbf{x}}-\mathbf{x}_k\Vert^2 + 2\langle \tilde{\mathbf{x}}-\mathbf{x}_k,\frac{1}{L_f}\nabla F_{\mu_{k}}(\mathbf{x}_k)\rangle + \Vert \frac{1}{L_f}\nabla F_{\mu_{k}}(\mathbf{x}_k)\Vert^2) + \beta{\mu_{k}} \\
 \leq & \beta{\mu_{k}}.
\end{split}
\end{equation*}
Here we used the definitions of $\Prog$, $\prox$ and $g_\mu$, as well as the convexity of $g_\mu$. Therefore from Lemma \ref{l:descent-direction} and backtracking condition (\ref{eq:line-search}) we obtain that if $k$ is a coarse correction step, then
\begin{equation*}
\begin{split}
L_f\Prog(\mathbf{x}_k) 
 & \leq \frac{1}{2}\Vert \nabla F_{\mu_{k}}(\mathbf{x}_k)\Vert^2 + L_f\beta{\mu_{k}} \\
 & \leq -\frac{L_H}{\kappa^2}\langle \mathbf{d}_k(\mathbf{x}_k),\nabla F_{\mu_{k}} (\mathbf{x}_k)\rangle + L_f\beta{\mu_{k}} \\
 & \leq \frac{L_H}{c s_k \kappa^2} ({F_{\mu_k}}(\mathbf{x}_k) - {F_{\mu_k}}(\mathbf{y}_{k+1})) + L_f\beta{\mu_{k}}.,
\end{split}
\end{equation*}
Otherwise, if $k$ is a gradient correction step, then from Lemma \ref{l:gradient-descent-guarantee}
\begin{equation*}
L_f\Prog(\mathbf{x}_k)\leq L_f (F(\mathbf{x}_k) - F(\mathbf{y}_{k+1})) \leq L_f ({F_{\mu_k}}(\mathbf{x}_k) - {F_{\mu_k}}(\mathbf{y}_{k+1})) + {L_f}\beta\mu_k.
\end{equation*}

Now choosing
\begin{equation*}
\eta_{k+1}=\left\{
    \begin{array}{ll}
      L_f,\quad\text{when}\; k+1\;\text{is a gradient correction step}\\
      \max\Big\{\frac{1}{4\alpha_k^2 \eta_k},\frac{L_H}{c s_k\kappa^2}
      \Big\},\quad\text{otherwise}\\
    \end{array}
  \right.
\end{equation*} 
we obtain the desired result.
\end{proof}

\begin{remark}
The recurrent choice of $\eta_{k+1}$ may seem strange at this point, as we could simply set it to $\max \{\frac{L_H}{c s_k\kappa^2}, L_f\}$, however forcing $\eta_{k+1}\geq \frac{1}{4\alpha_k^2 \eta_k}$ helps us ensure that $t_k\in (0,1]$ later.
\end{remark}

\begin{lemma}[Coupling] \label{l:coupling}
For any $\mathbf{u}\in\mathbb{R}^n$ and $ t_k=1/\alpha_{k+1}\eta_{k+1}$, where $\eta_{k+1}$ is defined as in Lemma \ref{l:magma-main-lemma}, it holds that
\begin{multline}
\alpha_{k+1}^2 \eta_{k+1} F(\mathbf{y}_{k+1})-(\alpha_{k+1}^2 \eta_{k+1}-\alpha_{k+1})F(\mathbf{y}_{k}) +(V_{\mathbf{z}_{k+1}}(\mathbf{u})-V_{\mathbf{z}_k}(\mathbf{u})) \\
\leq \alpha_{k+1} F(\mathbf{u}) + ({L_f + \eta_{k+1}})\alpha_{k+1}^2\beta{\mu_{k}}.
\end{multline}
\end{lemma}

\begin{proof}
First note that, if $k$ is a gradient correction step, then the proof follows from Lemma \ref{l:AGM-coupling}. Now assume $k$ is a coarse correction step, then using the convexity of $f$ we obtain
\begin{equation} \label{eq:coupling-1}
\begin{split}
F(\mathbf{x}_k)-F(\mathbf{u}) 
 & = f(\mathbf{x}_k)-f(\mathbf{u}) + g(\mathbf{x}_k)-g(\mathbf{u}) \\
 & \leq \langle \nabla f(\mathbf{x}_k),\mathbf{x}_k-\mathbf{u}\rangle + g(\mathbf{x}_k)-g(\mathbf{u}).
\end{split}
\end{equation}

On the other hand, from the definition of $g_{\mu_k}$ and its convexity we have
\begin{equation} \label{eq:coupling-2}
\begin{split}
g(\mathbf{x}_k)-g(\mathbf{z}_k) 
 & \leq g_{\mu_k}(\mathbf{x}_k) + \beta_1\mu_k - g_{\mu_k}(\mathbf{z}_k) + \beta_2\mu_k \\
 & \leq \langle \nabla g_{\mu_{k}}(\mathbf{x}_k),\mathbf{x}_k -\mathbf{z}_k\rangle + \beta{\mu_{k}}.
\end{split}
\end{equation}

Then using (\ref{eq:coupling-2}), we can {rewrite} $g(\mathbf{x}_k)-g(\mathbf{u})$ as
\begin{equation} \label{eq:coupling-3}
\begin{split}
g(\mathbf{x}_k)-g(\mathbf{u}) 
 & = (g(\mathbf{x}_k)-g(\mathbf{z}_k)) + (g(\mathbf{z}_k)-g(\mathbf{u})) \\
 & \leq \langle \nabla g_{\mu_{k}}(\mathbf{x}_k),\mathbf{x}_k -\mathbf{z}_k\rangle + \beta{\mu_{k}} + (g(\mathbf{z}_k)-g(\mathbf{u})).
\end{split}
\end{equation}

Now rewriting $\langle \nabla f(\mathbf{x}_k),\mathbf{x}_k-\mathbf{u}\rangle= \langle \nabla f(\mathbf{x}_k),\mathbf{x}_k-\mathbf{z}_k\rangle + \langle \nabla f(\mathbf{x}_k),\mathbf{z}_k-\mathbf{u}\rangle$ in (\ref{eq:coupling-1}) and using (\ref{eq:coupling-3}) we obtain
\begin{equation} \label{eq:coupling-4}
 F(\mathbf{x}_k)-F(\mathbf{u})
 \leq \langle \nabla F_{\mu_{k}}(\mathbf{x}_k),\mathbf{x}_k -\mathbf{z}_k\rangle + \beta\mu_k + \langle \nabla f(\mathbf{x}_k),\mathbf{z}_k-\mathbf{u}\rangle + (g(\mathbf{z}_k)-g(\mathbf{u})),
\end{equation}
where we used that $\nabla F_{\mu_k}(\mathbf{x}_k)=\nabla f(\mathbf{x}_k)+\nabla g_{\mu_k}(\mathbf{x}_k)$.

From the choice $ t_k(\mathbf{x}_k-\mathbf{z}_k)=(1- t_k)(\mathbf{y}_k-\mathbf{x}_k)$ and convexity of $F_{\mu_k}$ we have
\begin{equation*}
\begin{split}
\langle \nabla F_{\mu_{k}}(\mathbf{x}_k),\mathbf{x}_k -\mathbf{z}_k\rangle
 & = \frac{(1-t_k)}{t_k} \langle \nabla F_{\mu_{k}}(\mathbf{x}_k),\mathbf{y}_k -\mathbf{x}_k\rangle \\
 & \leq \frac{(1-t_k)}{t_k} (F_{\mu_k}(\mathbf{y}_k)-F_{\mu_k}(\mathbf{x}_k)) \\
 & \leq \frac{(1-t_k)}{t_k} (F(\mathbf{y}_k)-F(\mathbf{x}_k)) + \frac{(1-t_k)}{t_k}\beta\mu_k,
\end{split}
\end{equation*}
where for the last inequality we used that $F_{\mu_k}$ is a $\mu_k$-smooth approximation of $F$. Then choosing {$t_k=1/(\alpha_{k+1}\eta_{k+1})$}, where $\alpha_{k+1}$ is a step-size for the Mirror Descent step and $\eta_{k+1}$ is defined in Lemma \ref{l:magma-main-lemma}, we have
\begin{equation} \label{eq:coupling-5}
\langle \nabla F_{\mu_{k}}(\mathbf{x}_k),\mathbf{x}_k -\mathbf{z}_k\rangle
 \leq (\alpha_{k+1}\eta_{k+1} -1) (F(\mathbf{y}_k)-F(\mathbf{x}_k)) + (\alpha_{k+1} \eta_{k+1} - 1)\beta\mu_k.
\end{equation}

Plugging (\ref{eq:coupling-5}) into (\ref{eq:coupling-4}) we obtain
\begin{multline} \label{eq:coupling-6}
 \alpha_{k+1}(F(\mathbf{x}_k)-F(\mathbf{u})) 
 \leq (\alpha_{k+1}^2 \eta_{k+1} -\alpha_{k+1}) (F(\mathbf{y}_k)-F(\mathbf{x}_k)) + \alpha_{k+1}^2\eta_{k+1}\beta\mu_k \\
  + \alpha_{k+1}\langle \nabla f(\mathbf{x}_k),\mathbf{z}_k-\mathbf{u}\rangle + \alpha_{k+1}(g(\mathbf{z}_k)-g(\mathbf{u})).
\end{multline}

Finally, we apply Lemma \ref{l:magma-main-lemma} on (\ref{eq:coupling-6}) and obtain the desired result after simplifications:
\begin{multline*}
 \alpha_{k+1}(F(\mathbf{x}_k)-F(\mathbf{u})) \leq (\alpha_{k+1}^2\eta_{k+1}-\alpha_{k+1})(F(\mathbf{y}_k)-F(\mathbf{x}_k)) \\
 + \alpha_{k+1}^2 [\eta_{k+1}(F(\mathbf{x}_k) - F(\mathbf{y}_{k+1})) + ({L_f + \eta_{k+1}})\beta{\mu_{k}}] + V_{\mathbf{z}_k}(\mathbf{u})-V_{\mathbf{z}_{k+1}}(\mathbf{u}).
\end{multline*}

\end{proof}

\begin{theorem}[Convergence] \label{th:convergence}
After $T$ iterations of Algorithm \ref{alg:MAGMA}, without loss of generality assuming that the {last iteration} is a gradient correction step {for
\begin{equation} \label{eq:eta}
\eta_{k+1}=\left\{
                \begin{array}{ll}
                  L_f,\quad\text{when}\; k+1\;\text{is a gradient correction step}\\
                  \max\Big\{\frac{1}{4\alpha_k^2 \eta_k},\frac{L_H}{c s_k\kappa^2}
                  \Big\},\quad\text{otherwise}\\
                \end{array}
              \right.
\end{equation}
and
\begin{equation} \label{eq:alpha}
\alpha_{k+1}=\left\{
                \begin{array}{ll}
                  \frac{k+2}{2L_f},\quad\text{when}\; k+1\;\text{is a gradient correction step}\\
                  \frac{1}{2\eta_{k+1}}+\alpha_k\sqrt{\frac{\eta_k}{\eta_{k+1}}},\quad\text{otherwise}\\
                \end{array}
              \right.
\end{equation}
it holds that}
\begin{equation*}
F(\mathbf{y}_T)-F(\mathbf{x}^*)\leq \mathcal{O}\left(\frac{L_f}{T^2}\right),
\end{equation*}
\end{theorem}

\begin{proof}
{By choosing $\eta_{k+1}$ and $\alpha_{k+1}$ according to (\ref{eq:eta}) and (\ref{eq:alpha})}, we ensure that $\alpha_{k}^2 \eta_{k}=\alpha_{k+1}^2\eta_{k+1}-\alpha_{k+1}+\frac{1}{4\eta_{k+1}}$ and ${t_k=1/(\alpha_{k+1}\eta_{k+1})}\in(0,1]$. Then telescoping Lemma \ref{l:coupling} with $k=0,1,\ldots,T-1$ we obtain

\begin{multline*}
\sum_{k=0}^{T-1}[\alpha_{k+1}^2\eta_{k+1}F(\mathbf{y}_{k+1})-\alpha_{k}^2\eta_{k}F(\mathbf{y}_{k})+\frac{1}{4\eta_{k+1}}F(\mathbf{y}_k)]+V_{\mathbf{z}_T}(\mathbf{u})-V_{\mathbf{z}_0}(\mathbf{u}) \\
\leq \sum_{k=0}^{T-1}[\alpha_{k+1} F(\mathbf{u})+({L_f+\eta_{k+1}})\alpha_{k+1}^2\beta{\mu_{k}}].
\end{multline*}

Or equivalently,
\begin{multline*}
\alpha_T^2 \eta_T F(\mathbf{y}_T) - \alpha_0^2\eta_0 F(\mathbf{y}_0) + \sum_{k=0}^{T-1}\frac{1}{4\eta_{k+1}}F(\mathbf{y}_k)+V_{\mathbf{z}_T}(\mathbf{u})-V_{\mathbf{z}_0}(\mathbf{u}) \\
\leq \sum_{k=0}^{T-1}\left[\alpha_{k+1} F(\mathbf{u})+({L_f+\eta_{k+1}})\alpha_{k+1}^2\beta{\mu_{k}}\right],
\end{multline*}

Then choosing $\mathbf{u}=\mathbf{x}^*$ and noticing that $F(\mathbf{y}_k)\geq F(\mathbf{x}^*)$ and  $V_{\mathbf{z}_T}(\mathbf{u})\geq 0$, {for an upper bound $\Theta\geq V_{\mathbf{z}_0}(\mathbf{x}^*)$, we obtain}
\begin{equation*}
\alpha_T^2 \eta_T F(\mathbf{y}_T) - \alpha_0^2\eta_0 F(\mathbf{y}_0) \leq \Theta + \sum_{k=0}^{T-1}\left[\left(\alpha_{k+1} - \frac{1}{4\eta_{k+1}}\right) F(\mathbf{x}^*)+({L_f+\eta_{k+1}})\alpha_{k+1}^2\beta{\mu_{k}}\right].
\end{equation*}

Now using the fact that $\alpha_{k+1}-\frac{1}{4\eta_{k+1}} = \alpha_{k+1}^2\eta_{k+1}-\alpha_{k}^2 \eta_{k}$, for $k=0,1,\ldots,T-1$ we simplify to
\begin{equation*}
\alpha_T^2 \eta_T F(\mathbf{y}_T) - \alpha_0^2\eta_0 F(\mathbf{y}_0) \leq \Theta + \alpha_T^2 \eta_T F(\mathbf{x}^*) - \alpha_0^2\eta_0 F(\mathbf{x}^*) + \sum_{k=0}^{T-1}({L_f+\eta_{k+1}})\alpha_{k+1}^2\beta{\mu_{k}}.
\end{equation*}

{Then defining $\alpha_0=0$ and $\eta_0=L_f$ and using the property of Bregman distances we obtain}
\begin{equation*}
\begin{split}
\alpha_T^2 \eta_T (F(\mathbf{y}_T) - F(\mathbf{x}^*))
& \leq \Theta + \alpha_0^2\eta_0 (F(\mathbf{y}_0) - F(\mathbf{x}^*)) + \sum_{k=0}^{T-1}({L_f+\eta_{k+1}})\alpha_{k+1}^2\beta{\mu_{k}} \\
& \leq {\Theta + \frac{1}{L_f}} + \sum_{k=0}^{T-1}({L_f+\eta_{k+1}})\alpha_{k+1}^2\beta{\mu_{k}}.
\end{split}
\end{equation*}

Then assuming that the last iteration $T$ is a gradient correction step\footnote{This assumption does not lose the generality, as we can always perform one extra gradient correction iteration} and using the definitions of $\alpha_T$ and $\eta_T$ for the left hand side, we can further simplify to
\begin{equation*}
\frac{(T+1)^2}{4L_f}(F(\mathbf{y}_T)-F(\mathbf{x}^*)) \leq {\Theta} + \sum_{k=0}^{T-1}({L_f+\eta_{k+1}})\alpha_{k+1}^2\beta{\mu_{k}},
\end{equation*}
or equivalently
\begin{equation*}
F(\mathbf{y}_T)-F(\mathbf{x}^*) \leq 4L_f\frac{{\Theta} + \sum_{k=0}^{T-1}({L_f+\eta_{k+1}})\alpha_{k+1}^2\beta{\mu_{k}}}{(T+1)^2}.
\end{equation*}

Finally, choosing $\mu_{k}\in\Big(0, \frac{\zeta}{({L_f+\eta_{k+1}})\alpha_{k+1}^2\beta T}\Big]$ for a small predefined $\zeta\in (0,1]$, we obtain:
\begin{equation*}
F(\mathbf{y}_T)-F(\mathbf{x}^*) \leq {\frac{4L_f(\Theta+\zeta)}{(T+1)^2}}.
\end{equation*}
\end{proof}

\begin{remark}
Clearly, the constant factor of our Algorithm's worst case convergence rate is not better than that of AGM, however, as we show in the next section, in practice MAGMA can be several times faster than AGM.
\end{remark}

\section{Numerical Experiments} \label{sec:experiments}
In this section we demonstrate the empirical performance of our algorithm compared to FISTA \cite{beck2009fast} and a variant of MISTA \cite{parpas2014multilevel} on several large-scale face recognition problems. We chose those two particular algorithms, as MISTA is the only multi-level algorithm that can solve non-smooth composite problems and FISTA is the only other algorithm that was able to solve our large-scale problems in reasonable times. The source code and data we used in the numerical experiments is available from the webpage of the second author, \href{http://www.doc.ic.ac.uk/~pp500/magma.html}{www.doc.ic.ac.uk/$\sim$pp500/magma.html}. {In Section \ref{secRand1} we also compare FISTA and MAGMA to {two recent proximal stochastic algorithms: (APCG \cite{lin2015accelerated}) and SVRG \cite{xiao2014proximal}}. We chose Face Recognition (FR) as a demonstrative application since (a) FR using large scale dictionaries is a relatively unexplored problem \footnote{FR using large scale dictionaries is an unexplored problem in $\ell_1$ optimization literature due to the complexity of the (\ref{eq:L1regLS}) problem.} and (b) the performance of large scale face recognition depends on the face resolution.}

\subsection{Robust Face Recognition}
In \cite{wright2010dense} it was shown that a sparse signal $\mathbf{x}\in\mathbb{R}^n$ from highly corrupted linear measurements $\mathbf{b}=\mathbf{A}\mathbf{x}+\mathbf{e}\in\mathbb{R}^n$, where $\mathbf{e}$ is an unknown error and $\mathbf{A}\in\mathbb{R}^{m\times n}$ is a highly correlated dictionary, can be recovered by solving the following $\ell_1$-minimization problem
\begin{equation*}
\begin{array}{ccc}
\min\limits_{\mathbf{x}\in\mathbb{R}^n ,\mathbf{e}\in\mathbb{R}^m}\Vert\mathbf{x}\Vert_{1} + \Vert\mathbf{e}\Vert_{1}  & \text{subject to} & \mathbf{A}\mathbf{x}+\mathbf{e}=\mathbf{b}.
\end{array}
\end{equation*}
It was shown that accurate recovery of sparse signals is possible and computationally feasible even for images with nearly $100\%$ of the observations corrupted. 

We introduce the new variable $\mathbf{w}=[\mathbf{x}^\top, \mathbf{e}^\top]^\top \in\mathbb{R}^{n+m}$ and matrix $\mathbf{B}=[\mathbf{A},\mathbf{I}] \in\mathbb{R}^{m\times(m+n)}$, where $\mathbf{I}\in\mathbb{R}^{m\times m}$ is the identity matrix. The updated (\ref{eq:L1regLS}) problem can be written as the following optimization problem:

\begin{equation} \label{eq:L1regLS-CAB}
\min_{\mathbf{w}\in\mathbb{R}^{m+n}}{\frac{1}{2}\Vert \mathbf{B} \mathbf{w}-\mathbf{b}\Vert_2^2+\lambda\Vert \mathbf{w}\Vert_1}.
\end{equation}

A popular application of dense error correction in computer vision, is the face recognition problem. There $\mathbf{A}$ is a dictionary of facial images stacked as column vectors\footnote{Facial images are, indeed, highly correlated.}, $\mathbf{b}$ is an incoming image of a person who we aim to identify in the dictionary and the non-zero entries of the sparse solution $\mathbf{x}^*$ (obtained from the solution of (\ref{eq:L1regLS-CAB}) $\mathbf{w}^*=[{\mathbf{x}^*}^\top,{\mathbf{e}^*}^\top]^\top$) indicate the images that belong to the same person as $\mathbf{b}$. As stated above, $\mathbf{b}$ can be highly corrupted, e.g. with noise and/or occlusion.

In a recent survey Yang et. al. \cite{yang2013fast} compared a number of algorithms solving the face recognition problem with regard to their efficiency and recognition rates. Their experiments showed that the best algorithms were the Dual Proximal Augmented Lagrangian Method (DALM), primal dual interior point method (PDIPA) and L1LS \cite{kim2007interior}, however, the images in their experiments were of relatively small dimension - $40\times 30$ pixels. Consequently the problems were solved within a few seconds. It is important to notice that both DALM and PDIPA are unable to solve large-scale (both large number of images in $\mathbf{A}$ and large image dimensions) problems, as the former requires inverting $\mathbf{A}\mathbf{A}^\top$ and the later uses Newton update steps. On the other hand L1LS is designed to take advantage of sparse problems structures, however it performs significantly worse whenever the problem is not very sparse.

\subsection{MAGMA for Solving Robust Face Recognition}
In order to be able to use the multi-level algorithm, we need to define a coarse model and restriction and prolongation operators appropriate for the given problem. In this subsection we describe those constructions for the face recognition problem used in our numerical experiments.

Creating a coarse model requires decreasing the size of the original fine level problem. In our experiments we only reduce $n$, as we are trying to improve over the performance of first order methods, {whose complexity per iteration is} $\mathcal{O}(n^2)$. Also the columns of $\mathbf{A}$ are highly correlated, which means that reducing $n$ loses little information about the original problem, whereas it is not necessarily true for the rows of $\mathbf{A}$.  Therefore, we introduce the dictionary $\mathbf{A}_H\in\mathbb{R}^{m\times n_H}$ with $n_H<n$. More precisely, we set $\mathbf{A}_H=\mathbf{A} \mathbf{R}_x^\top$, with $\mathbf{R}_x$ defined in (\ref{eq:Rx}). With the given restriction operator we are ready to construct a coarse model for the problem (\ref{eq:L1regLS-CAB}):
\begin{equation} \label{eq:L1regLS-CAB-coarse}
\min_{\mathbf{w}_H\in\mathbb{R}^{m+n_H}}{\frac{1}{2}\Vert \begin{bmatrix}
\mathbf{A}_H & \mathbf{I}
\end{bmatrix}\mathbf{w}_H-\mathbf{b}\Vert_2^2
+\lambda\sum_{j=1}^{m+n_H}\sqrt{\mu_H^2+\mathbf{w}_{H,j}^2}
+\langle \mathbf{v}_H,\mathbf{w}_H\rangle},
\end{equation}
where $\mathbf{w}_H=[\mathbf{x}_H^\top, \mathbf{e}^\top]^\top$ and $\mathbf{v}_H$ is defined in (\ref{eq:v-term}). It is easy to check that the coarse objective function and gradient can be evaluated using the following equations:
\begin{equation*}
F_H(\mathbf{w}_H)
=\frac{1}{2}\Vert \mathbf{A}_H \mathbf{x}_H+\mathbf{e}-\mathbf{b}\Vert_2^2
+\lambda\sum_{j=1}^{m+n_H}\sqrt{\mu_H^2+\mathbf{w}_{H,j}^2}
+\langle \mathbf{v}_H,\mathbf{w}_H\rangle,
\end{equation*}
\begin{equation*}
\nabla F_H(\mathbf{w}_H)
=\begin{bmatrix}
\mathbf{A}_H^\top \mathbf{A}_H & \mathbf{A}_H^\top\\
\mathbf{A}_H & \mathbf{I}
\end{bmatrix}\begin{bmatrix}
\mathbf{x}_H\\
\mathbf{e}
\end{bmatrix}-\begin{bmatrix}
\mathbf{A}_H^\top \mathbf{b}\\
\mathbf{b}
\end{bmatrix}
+\nabla g_H (\mathbf{w}_H)
+\mathbf{v}_H,
\end{equation*}
where $\nabla g_H (\mathbf{w}_H)$ is the gradient of $g_H $ defined in (\ref{eq:g_mu}) and is given elementwise as follows:
\begin{equation*}
\frac{\lambda\mathbf{w}_{H,j}}{\sqrt{\mu_H^2+\mathbf{w}_{H,j}^2}}, \forall j=1,2,\ldots,n_H.
\end{equation*}
Hence we do not multiply matrices of size $m\times (m+n_H)$, but only $m\times n_H$. 

We use a standard full weighting operator \cite{briggs2000multigrid} as a restriction operator,
\begin{equation} \label{eq:Rx}
\mathbf{R}_x = \frac{1}{4}\begin{bmatrix}
2 & 1 & 0 & 0 & 0 & 0 &...& 0 \\
0 & 1 & 2 & 1 & 0 & 0 &...& 0 \\
0 & 0 & 0 & 1 & 2 & 1 &...& 0 \\
  &   &	  &...&   &   &   & 0 \\
0 &...&   & 0 & 0 & 1 & 2 & 1 \end{bmatrix}{\in\mathbb{R}^{\frac{n_i}{2}\times n_i}}
\end{equation}
and its transpose for the prolongation operator {at each level $i\in\{1,...,H-1\}$, where $n_1=n$, $n_{i+1}=\frac{n_i}{2}$ and $H$ is the depth of each V-cycle}. However, we do not apply those operators on the whole variable $\mathbf{w}$ of the model (\ref{eq:L1regLS-CAB}), as this would imply also reducing $m$. We rather apply the operators only on part $\mathbf{x}$ of $\mathbf{w}=[\mathbf{x}^\top, \mathbf{e}^\top]^\top$, therefore our restriction and prolongation operators are of the following forms: $\mathbf{R}=[\mathbf{R}_x, \mathbf{I}]$ and $\mathbf{P}=\mathbf{R}^\top=[\mathbf{R}_x^\top, \mathbf{I}]^\top$.

{In all our experiments, for the $\Mirr$ operator we chose the standard Euclidean norm $\Vert \cdot\Vert_2$ and accordingly $\frac{1}{2}\Vert \cdot\Vert_2^2$ as a Bregman divergence.}

\subsection{Numerical Results} \label{sec:experimental results}

The proposed algorithm has been implemented in MatLab\textregistered\quad and tested on a PC with Intel\textregistered\quad Core\texttrademark\quad i7 CPU (3.4GHz$\times$8) and 15.6GB memory.

In order to create a large scale FR setting we created a dictionary $\mathbf{A}$ of up to $8,824$ images from more than $4,000$ different people \footnote{Some people had up to 5 facial images in the database some only one.} by merging images from several facial databases captured in controlled and in uncontrolled conditions. In particular, we used images from MultiPIE \cite{gross2010multi}, XM2VTS \cite{luettin1998evaluation} and many in-the-wild databases such as LFW \cite{huang2007labeled}, LFPW \cite{belhumeur2011localizing} and HELEN \cite{le2012interactive}. In \cite{yang2013fast} the dictionary was small hence only very low-resolution images were considered. However, in large scale FR having higher resolution images is very beneficial (e.g., from a random cohort of 300 selected images, low resolution images of $40\times 30$ achieved around $85\%$ recognition rate, while using images of $200\times 200$ we went up to $100\%$). Hence, in the remaining of our experiments the dictionary used is $\begin{bmatrix}\mathbf{A} & \mathbf{I}\end{bmatrix}$ of $40,000 \times 48,824$ and some subsets of this dictionary. 

In order to show the improvements of our method with regards to FISTA as well as to MISTA, we have designed the following experimental settings
\begin{itemize}
\item Using 440 facial images, then $\begin{bmatrix}\mathbf{A} & \mathbf{I}\end{bmatrix}$ is of $40,000 \times 40,440$
\item Using 944 facial images, then $\begin{bmatrix}\mathbf{A} & \mathbf{I}\end{bmatrix}$ is of $40,000 \times 40,944$
\item Using 8,824 facial images, then $\begin{bmatrix}\mathbf{A} & \mathbf{I}\end{bmatrix}$ is of $40,000 \times 48,824$
\end{itemize}

For each dictionary we randomly chose $20$ different input images $\mathbf{b}$ that are not in the dictionary and ran the algorithms with $4$ different random starting points. For all of those experiments we set the parameters of the algorithms as follows: $\kappa=0.9$, $0.8$ and $0.7$ respectively for experiment settings with $440$, $944$ and $8,824$ images in the database, $K_d=30$ for the two smaller settings and $K_d=50$ for the largest one, {the convergence tolerance was set to} $\epsilon=10^{-6}$ for the two smaller experiments and $\epsilon=10^{-7}$ for the largest one, $\lambda=10^{-6}$, $\tau=0.95$ and $s^0=10$. For larger dictionaries we slightly adjust $\kappa$ and $K_d$ in order to adjust for the larger problem size, giving the algorithm more freedom to perform coarse iterations. In all experiments we solve the coarse models with lower tolerance than the original problem, namely with tolerance $10^{-3}$.

{For all experiments we used as small coarse models as possible so that the corresponding algorithm could produce sparse solutions correctly identifying the sought faces in the dictionaries. Specifically, for experiments on the largest database with $8824$ images we used $13$ levels for MAGMA (so that in the coarse model $\mathbf{A}_H$ has only $2$ columns) and only two levels for MISTA, since with more than two levels MISTA was unable to produce sparse solutions. The number of levels used in MAGMA and MISTA are tabulated in Table \ref{t:coarse-levels}.}

\begin{table}
\center
{
\begin{tabular}{|c|c|c|c|}
	\hline 
	      & $440$ images & $944$ images & $8824$ images \\
	\hline 
	MISTA & 7 & 2 & 2 \\ 
	\hline 
	MAGMA & 7 & 7 & 13 \\ 
	\hline 
\end{tabular}}
\caption{{Number of coarse levels used for MISTA and MAGMA for each dictionary}}
\label{t:coarse-levels}
\end{table}

In all experiments we measure and compare the CPU running times of each tested algorithm, because comparing the number of iterations of a multi-level method against a standard first order algorithm would not be fair. This is justified as the multilevel algorithm may need fewer iterations on the fine level, but with a higher computational cost for obtaining each coarse direction. Comparing all iterations including the ones in coarser levels would not be fair either, as those are noticeably cheaper than the ones on the fine level. Even comparing the number of objective function values and gradient evaluations would not be a good solution, as on the coarse level those are also noticeably cheaper than on the fine level. Furthermore, the multilevel algorithm requires additional computations for constructing the coarse model, as well as for vector restrictions and prolongations. Therefore, we compare the algorithms with respect to CPU time.

Figures \ref{p:exp1}, \ref{p:exp2} and \ref{p:exp3} show the relative CPU running times until convergence of MAGMA and MISTA compared to FISTA\footnote{We only report the performance of FISTA, as AGM and FISTA demonstrate very similar performances on all experiments.}. The horizontal axis indicate different input images $\mathbf{b}$, whereas the vertical lines on the plots show the respective highest and lowest values obtained from different starting points. Furthermore, in Table \ref{t:average-times} we present the average CPU times required by each algorithm until convergence for each problem setting, namely with databases of $440$, $944$ and $8824$ images. As the experimental results show, MAGMA is $2-10$ times faster than FISTA. On the other hand, MISTA is more problem dependant. On some instances it can be even faster than MAGMA, but on most cases its performance is somewhere in between FISTA and MAGMA.


\begin{figure}
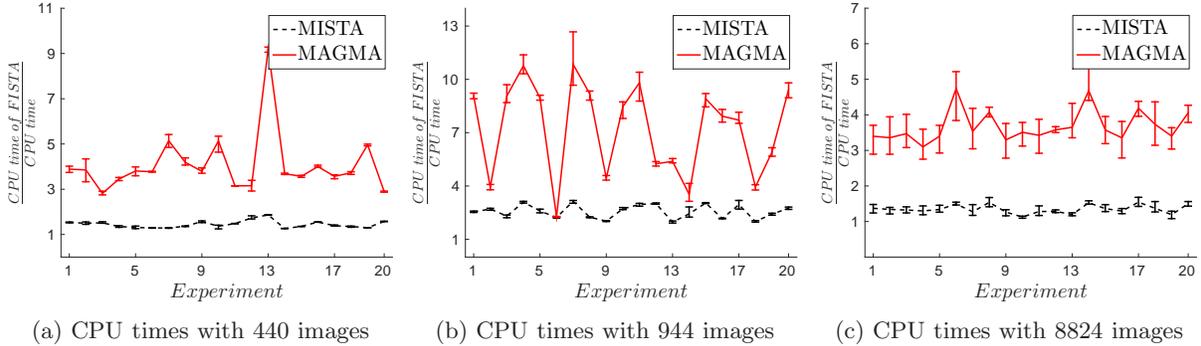

	\centering
	\begin{subfigure}[b]{0.3\textwidth}
	    \centering
		\includegraphics[scale=0.28]{time_exp1}
		\caption{CPU times with $440$ images}
		\label{p:exp1}
	\end{subfigure}
	\begin{subfigure}[b]{0.3\textwidth}
	    \centering
		\includegraphics[scale=0.28]{time_exp2}
		\caption{CPU times with $944$ images}
		\label{p:exp2}
	\end{subfigure}	
	\begin{subfigure}[b]{0.3\textwidth}
	    \centering
		\includegraphics[scale=0.28]{time_exp3}
		\caption{CPU times with $8824$ images}
		\label{p:exp3}
	\end{subfigure}
	\caption{\textbf{Comparing MAGMA and MISTA with FISTA}. Relative CPU times of MAGMA and MISTA compared to FISTA on face recognition experiments with dictionaries of varying number of images ($440$, $944$ and $8824$). For each dictionary $20$ random input images $\mathbf{b}$ were chosen, which indicate the numbers on horizontal axis. Additionally, the results with $4$ random starting points for each experiment are reflected as horizontal error bars.}
\end{figure}

\begin{table}
\center
\begin{tabular}{|c|c|c|c|}
	\hline 
	& $440$ images & $944$ images & $8824$ images \\
	\hline 
	FISTA & 98.69 & 175.77 & 1753 \\ 
	\hline 
	MISTA & 68.77 & 70.4 & 1302 \\ 
	\hline 
	MAGMA & 25.73 & 29.62 & 481 \\ 
	\hline 
\end{tabular} 
\caption{Average CPU running times (seconds) of MISTA, FISTA and MAGMA}
\label{t:average-times}
\end{table}

In order to better understand the experimental convergence speed of MAGMA, we measured the error of stopping criteria\footnote{We use the norm of the gradient mapping as a stopping criterion} and objective function values generated by MAGMA and FISTA over time from an experiment with $440$ images in the dictionary. The stopping criteria are log-log-plotted in Figure \ref{p:exp3-crit-log} and the objective function values - in Figure \ref{p:exp3-fval-log}. In all figures the horizontal axis is the CPU time in seconds (in log scale). Note that the coarse correction steps of MAGMA take place at the steep dropping points giving large decrease, but then at gradient correction steps first large increase and then  relatively constant behaviour is recorded for both plots. Overall, MAGMA reduces both objective value and the norm of the gradient mapping significantly faster.

\begin{figure}
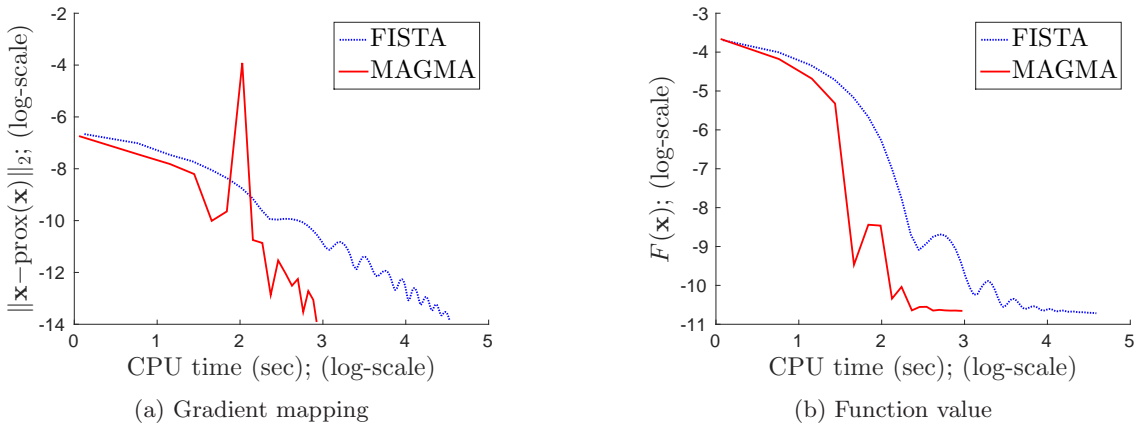

	\centering
	\begin{subfigure}[b]{0.48\textwidth}
	    \centering
		\includegraphics[scale=0.35]{steps_exp1_crit_log}
		\caption{Gradient mapping}
		\label{p:exp3-crit-log}
	\end{subfigure}
	\begin{subfigure}[b]{0.48\textwidth}
	    \centering
		\includegraphics[scale=0.35]{steps_exp1_fval_log}
		\caption{Function value}
		\label{p:exp3-fval-log}
	\end{subfigure}
	
	\caption{\textbf{MAGMA and FISTA convergences over time}. Comparisons of gradient mappings and function values of MAGMA, MISTA and FISTA over time.}
\end{figure}

\subsection{Stochastic Algorithms for Robust Face Recognition}\label{secRand1}

{Since (randomized) coordinate descent (CD) {and stochastic gradient (SG)} algorithms are considered as state of the art for general $\ell_1$ regularized least squared problems, we finish this section with a discussion on the application of these methods to the {robust} face recognition problem.} {We implemented {two} recent {algorithms:} accelerated randomised proximal coordinate gradient method (APCG) \footnote{We implemented the efficient version of the algorithm that avoids full vector operations.} \cite{lin2015accelerated} {and proximal stochastic variance reduced gradient (SVRG) \cite{xiao2014proximal}}. We also tried block coordinate methods with cyclic updates \cite{beck2013convergence}, however the randomised version of block coordinate method performs significantly better, hence we show only the performances of APCG {and SVRG} so as  not to clutter the results.}

{In our numerical experiments we found that  CD and {SG} algorithms are not suitable for {robust} face recognition problems. There are two reasons for this. Firstly, the data matrix contains highly correlated data. The correlation is due to the fact that we are dealing with a fine-grained object recognition problem. That is, the samples of the dictionary are all facial images of different people and the problem is to identify the particular individual.} {The second reason is the need to extend the standard $\ell_1$ regularized least squared model so that it can handle gross errors, such as occlusions. It can be achieved by using the dense error correction formulation \cite{wright2010dense} in (\ref{eq:L1regLS-CAB}) (also referred to as   the ``bucket" model in  \cite{wright2010dense}).}

{To demonstrate our argument we implemented APCG {and SVRG} and compared {them} with FISTA \cite{beck2009fast} and our method - MAGMA. We run all {four} algorithms for a fixed time and compare the achieved function values. For APCG we tried three different block size strategies, namely $1$, $10$ and $50$. While $10$ and $50$ usually performed similarly, $1$ was exceptionally slow, so we report the best achieved results for all experiments.} {For SVRG one has to choose a fixed step size as well as a number of inner iterations. We tuned both parameters to achieve the best possible results for each particular problem.}

{We performed the experiments on $5$ databases: the $3$ reported in the previous subsection (with $440$, $944$ and $8824$ images of $200\times 200$ dimension) and two ``databases" with data generated uniformly at random with dimensions $40,000\times 100$ and $40,000\times 8,000$. The later experiments are an extreme case where full gradient methods (FISTA) are outperformed by {both} APCG and {SVRG} when on standard $\ell_1$-regularized least squared problems, while MAGMA is not applicable. The results are given in Figure \ref{f:apcg_fista}.} { As we can see from Figures \ref{p:dic-2-val-nobucket} and \ref{p:dic-1-val-nobucket} all three methods can achieve low objective values very quickly for standard $\ell_1$ regularized least squared problems. However, after adding an identity matrix to the right of database (dense error correction model) the performance of all algorithms changes: they all become significantly slower due to larger problem dimensions (Figures \ref{p:dic-2-val} and \ref{p:dic-1-val}). Most noticeably SVRG stagnates after first one-two iterations. For the smaller problem (Figure \ref{p:dic-2-val}) FISTA is the fastest to converge to a minimizer, but for the larger problem (Figure \ref{p:dic-1-val}) APCG performs best. The picture is quite similar when looking at optimality conditions instead of objective function values in Figures \ref{p:dic-2-crit-nobucket} to \ref{p:dic-1-crit}: all three algorithms, especially SVRG are slower on the dense error correction model and for the largest problem APCG performs best.}

{This demonstrates that for $\ell_1$-regularized least squares problems {partial gradient methods} with a random dictionary can often be faster than accelerated full gradient methods. However, for problems with highly correlated dictionaries and  {noisy data}, such as the  {robust} face recognition problem, the picture is quite different.}

\begin{figure}
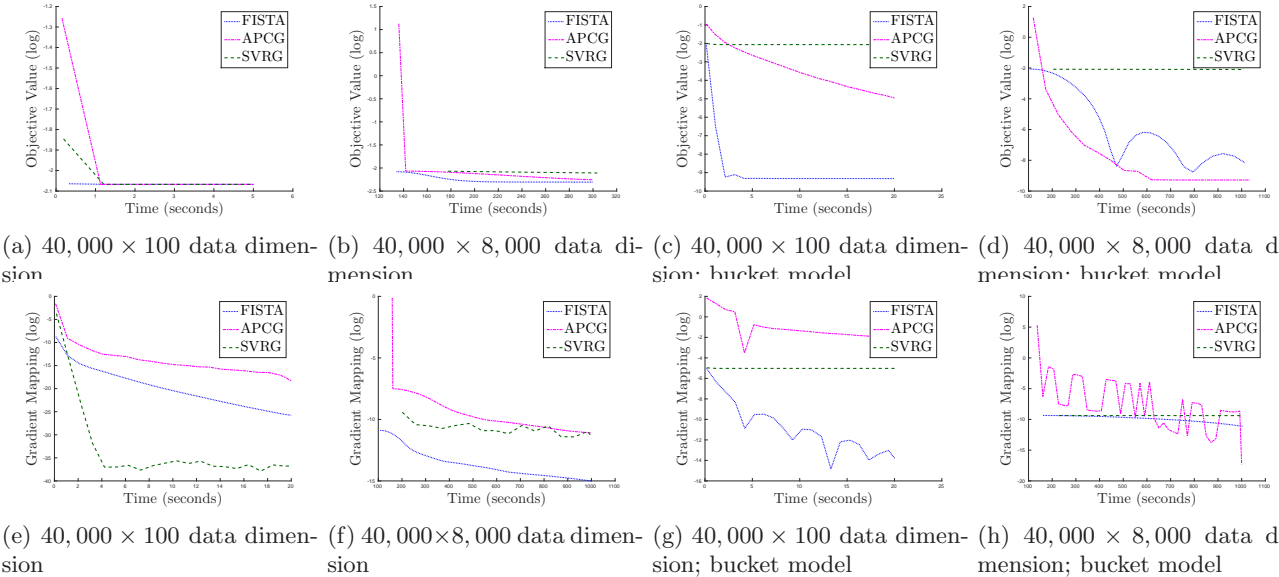

	\centering
	\begin{subfigure}[b]{0.24\textwidth}
	    \centering
		\includegraphics[scale=0.2]{dic-2-val-nobucket}
		\caption{$40,000\times 100$ data dimension}
		\label{p:dic-2-val-nobucket}
	\end{subfigure}
	\begin{subfigure}[b]{0.24\textwidth}
	    \centering
		\includegraphics[scale=0.2]{dic-1-val-nobucket}
		\caption{$40,000\times 8,000$ data dimension}
		\label{p:dic-1-val-nobucket}
	\end{subfigure}
	\begin{subfigure}[b]{0.24\textwidth}
	    \centering
		\includegraphics[scale=0.2]{dic-2-val}
		\caption{$40,000\times 100$ data dimension; bucket model}
		\label{p:dic-2-val}
	\end{subfigure}
	\begin{subfigure}[b]{0.24\textwidth}
	    \centering
		\includegraphics[scale=0.2]{dic-1-val}
		\caption{$40,000\times 8,000$ data dimension; bucket model}
		\label{p:dic-1-val}
	\end{subfigure}

	\begin{subfigure}[b]{0.24\textwidth}
	    \centering
		\includegraphics[scale=0.2]{dic-2-crit-nobucket}
		\caption{$40,000\times 100$ data dimension}
		\label{p:dic-2-crit-nobucket}
	\end{subfigure}
	\begin{subfigure}[b]{0.24\textwidth}
	    \centering
		\includegraphics[scale=0.2]{dic-1-crit-nobucket}
		\caption{$40,000\times 8,000$ data dimension}
		\label{p:dic-1-crit-nobucket}
	\end{subfigure}
	\begin{subfigure}[b]{0.24\textwidth}
	    \centering
		\includegraphics[scale=0.2]{dic-2-crit}
		\caption{$40,000\times 100$ data dimension; bucket model}
		\label{p:dic-2-crit}
	\end{subfigure}
	\begin{subfigure}[b]{0.24\textwidth}
	    \centering
		\includegraphics[scale=0.2]{dic-1-crit}
		\caption{$40,000\times 8,000$ data dimension; bucket model}
		\label{p:dic-1-crit}
	\end{subfigure}
	\caption{{Comparing FISTA, APCG and {SVRG} after running them for a fixed amount of time {with ((c), (d), (g) and (h)) and without ((a), (b), (e) and (f)) the bucket model} on randomly generated data.} {The first row contains sequences of the objective values and the bottom row contains optimality criteria, i.e gradient mappings.}}
	\label{f:apcg_fista}
\end{figure}

{Having established that APCG {and SVRG are} superior than FISTA for random data {with no gross errors}, we turn our attention to the problem at hand, i.e. the {robust} face recognition problem.} {First we run FISTA, APCG, SVRG and MAGMA on problems with no noise and thus the bucket model (\ref{eq:L1regLS-CAB}) is not used. We run all algorithms for $100$ seconds on a problem with $n=440$ images of $m=200\times 200=40,000$ dimension in the database. As the results in Figure \ref{f:fista-svrg-magma} show all algorithms achieve very similar solutions of good quality with APCG resulting in obviously less sparse solution (Figure \ref{p:apcg}).}

\begin{figure}
	\centering
	\begin{subfigure}[b]{0.19\textwidth}
	    \centering
		\includegraphics[scale=0.25]{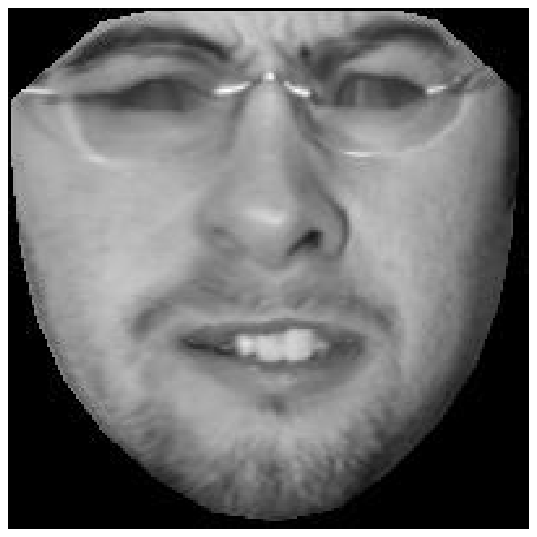}
		\caption{Original}
		\label{p:orig}
	\end{subfigure}
	\begin{subfigure}[b]{0.19\textwidth}
	    \centering
		\includegraphics[scale=0.25]{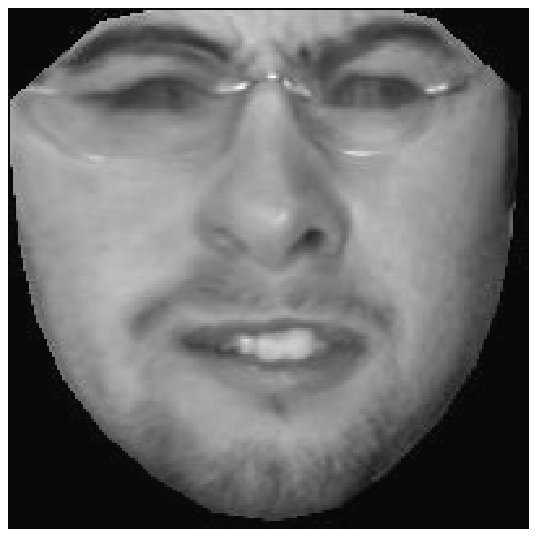}
		\caption{FISTA}
		\label{p:fista-img}
	\end{subfigure}
	\begin{subfigure}[b]{0.19\textwidth}
	    \centering
		\includegraphics[scale=0.25]{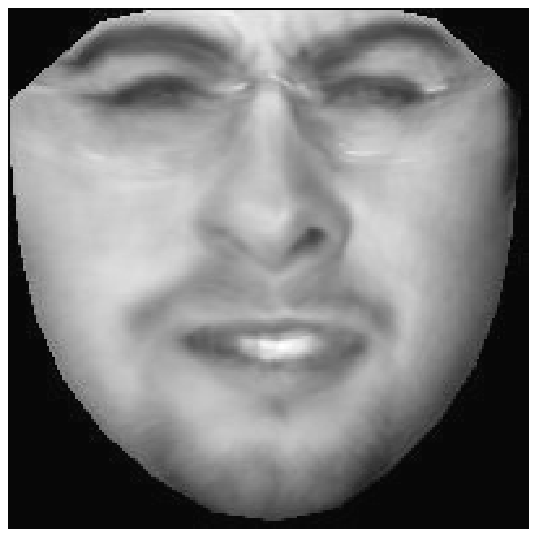}
		\caption{APCG}
		\label{p:apcg-img}
	\end{subfigure}
	\begin{subfigure}[b]{0.19\textwidth}
	    \centering
		\includegraphics[scale=0.25]{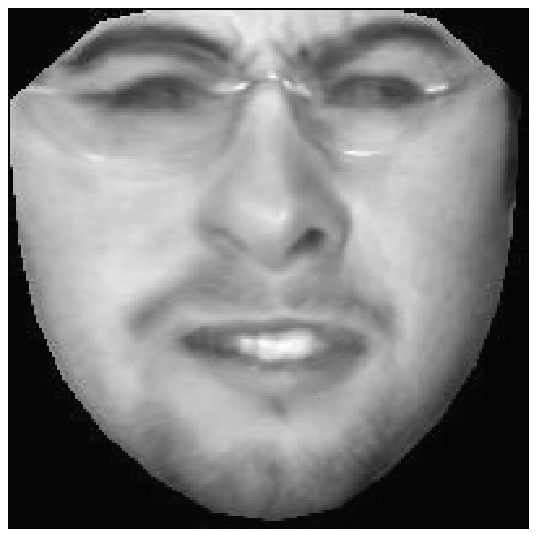}
		\caption{SVRG}
		\label{p:svrg-img}
	\end{subfigure}
	\begin{subfigure}[b]{0.19\textwidth}
	    \centering
		\includegraphics[scale=0.25]{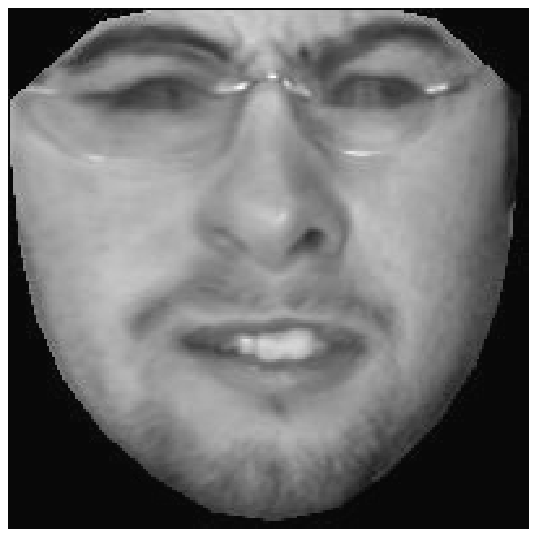}
		\caption{MAGMA}
		\label{p:magma-img}
	\end{subfigure}
	
	\begin{subfigure}[b]{0.24\textwidth}
	    \centering
		\includegraphics[scale=0.25]{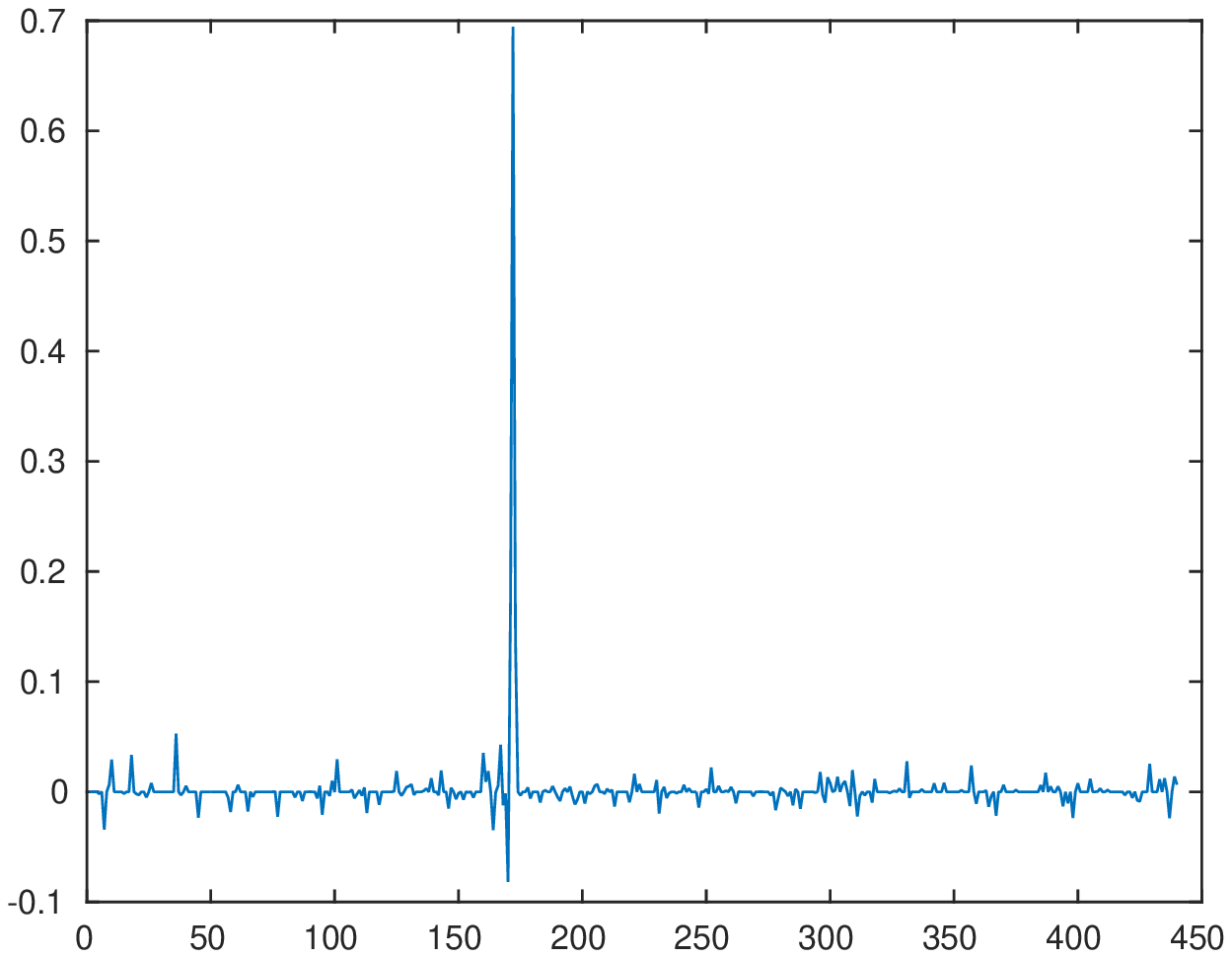}
		\caption{FISTA}
		\label{p:fista}
	\end{subfigure}
	\begin{subfigure}[b]{0.24\textwidth}
	    \centering
		\includegraphics[scale=0.25]{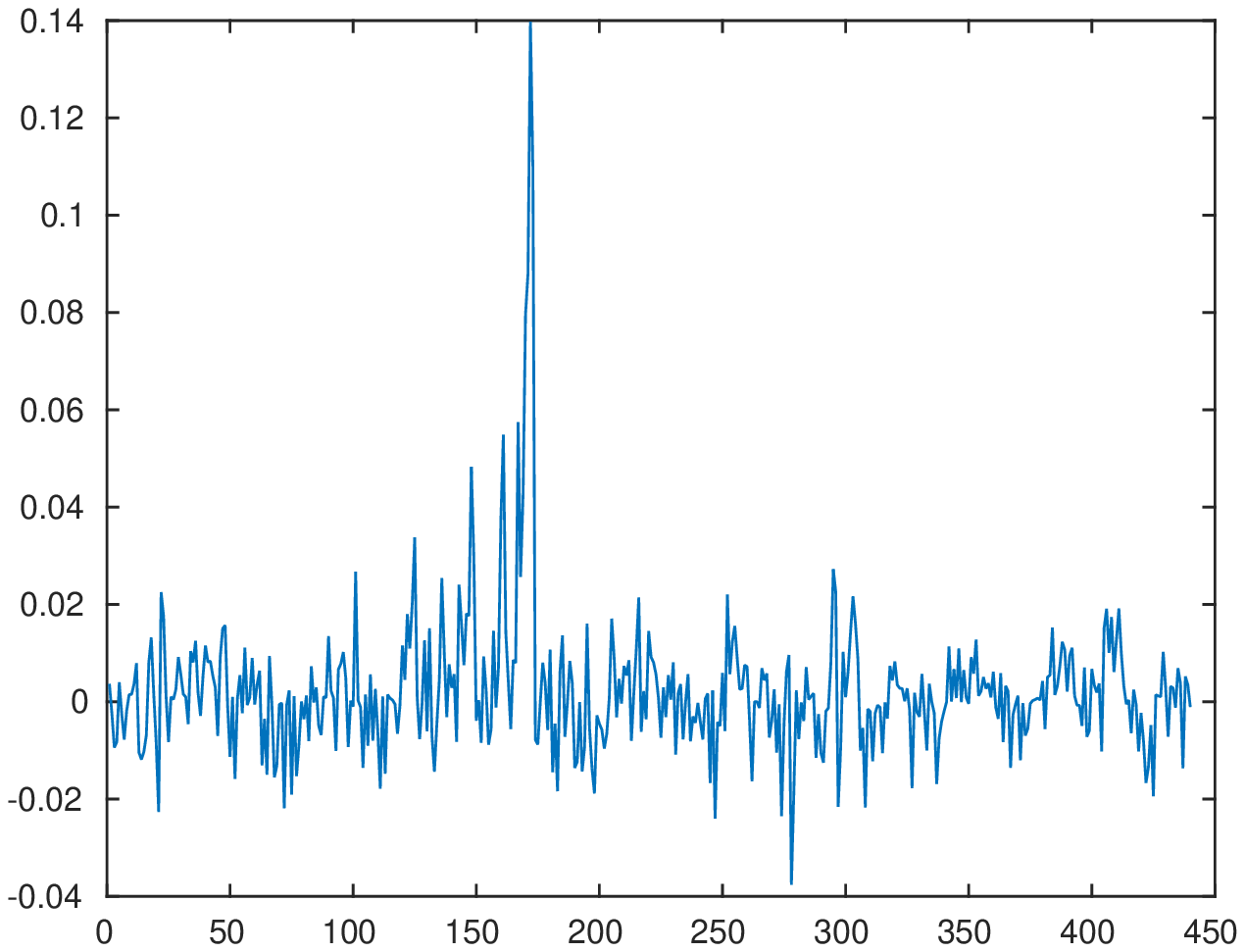}
		\caption{APCG}
		\label{p:apcg}
	\end{subfigure}
	\begin{subfigure}[b]{0.24\textwidth}
	    \centering
		\includegraphics[scale=0.25]{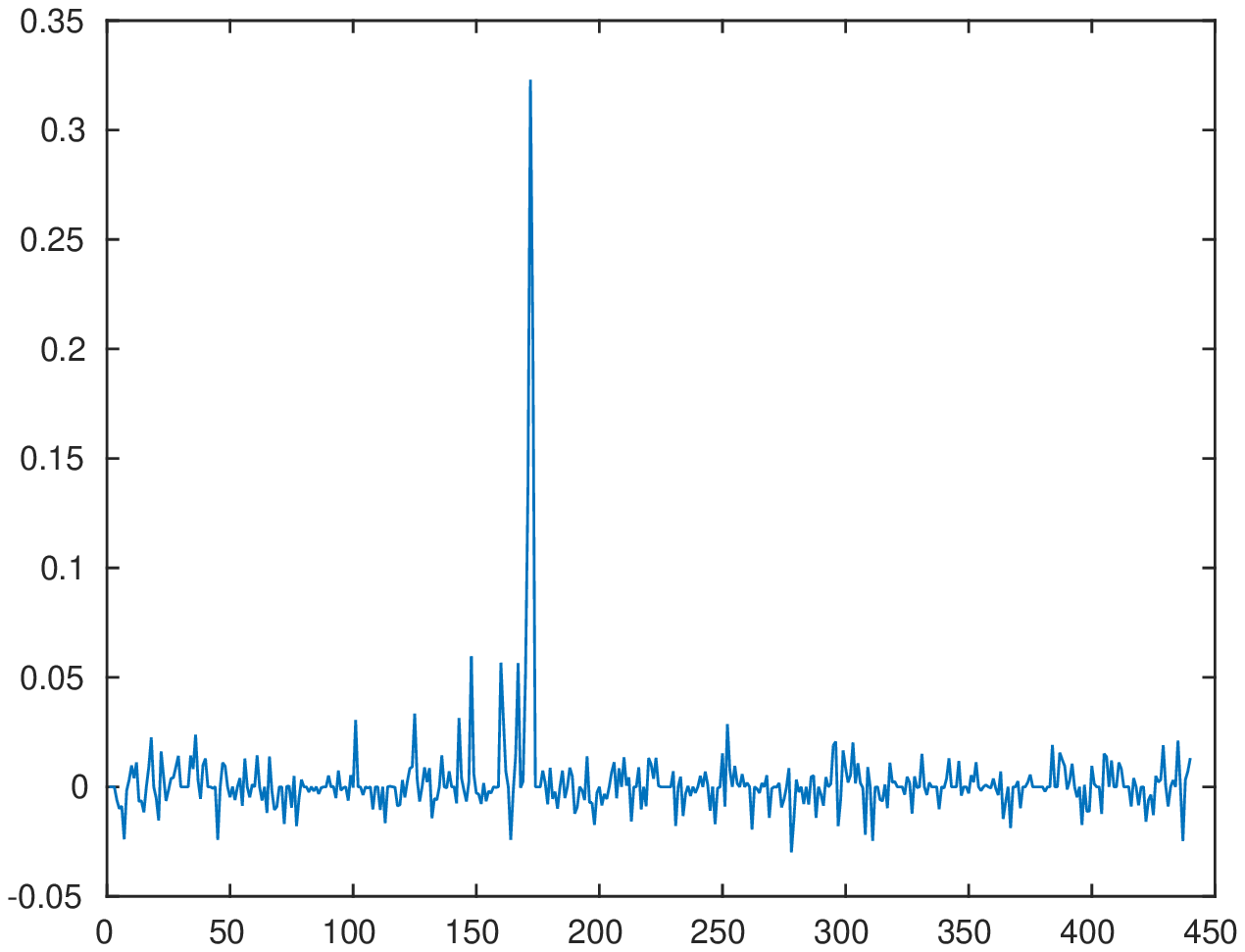}
		\caption{SVRG}
		\label{p:svrg}
	\end{subfigure}
	\begin{subfigure}[b]{0.24\textwidth}
	    \centering
		\includegraphics[scale=0.25]{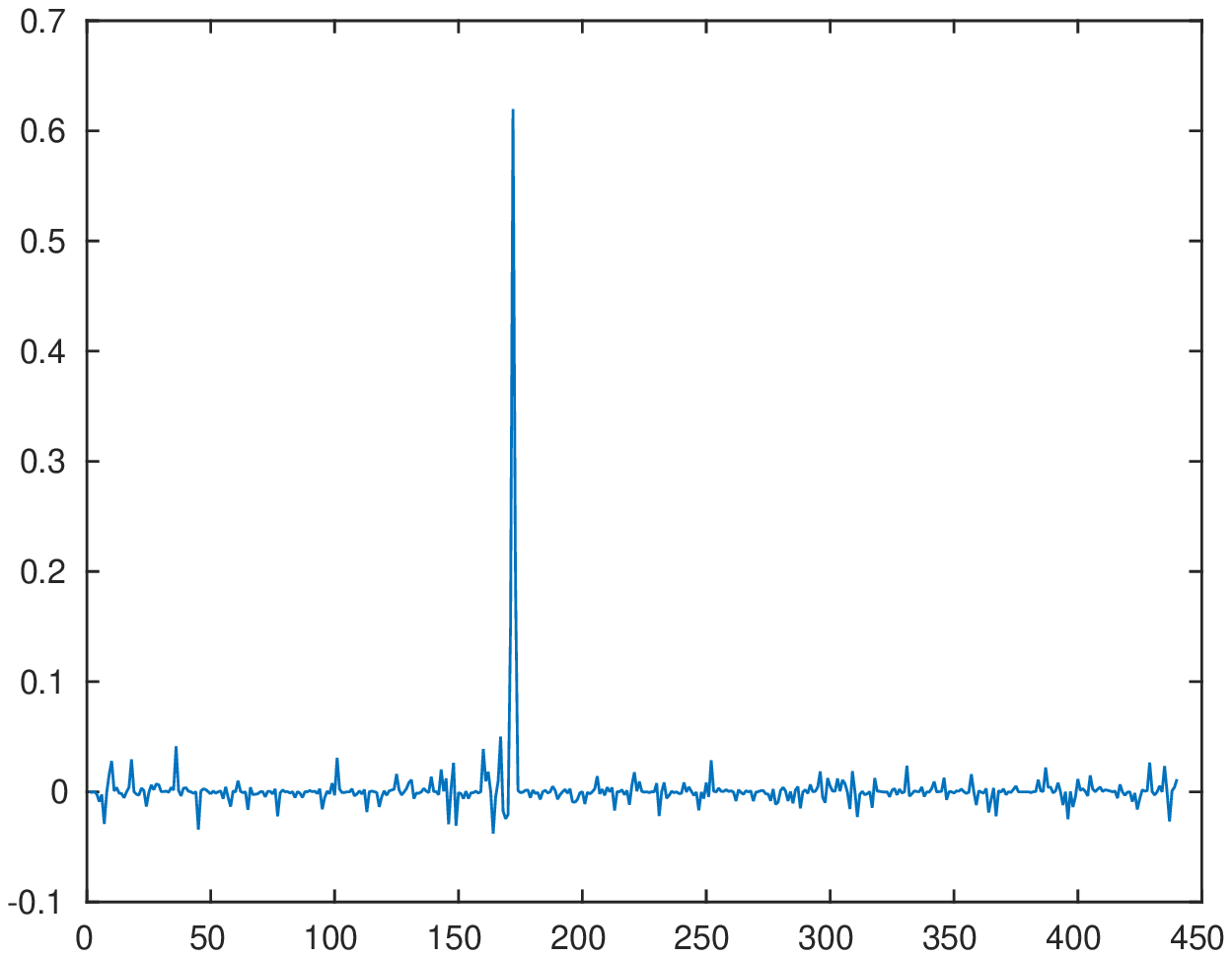}
		\caption{MAGMA}
		\label{p:magma}
	\end{subfigure}
	\caption{{Result of FISTA, APCG, SVRG and MAGMA in an image decomposition example without occlusion. Top row ((a) to (e))  contains the original input image and reconstructed images from each algorithm. The bottom row ((f) to (i)) contains the solutions $\mathbf{x}^\star$ from each corresponding algorithm.}}
	\label{f:fista-svrg-magma}
\end{figure}

{Then we run all the tested algorithms on the same problem, but this time we simulated occlusion by adding a black square on the incoming image. Thus, we need to solve the dense error correction optimisation problem (\ref{eq:L1regLS-CAB}). The results are shown in Figure \ref{f:fista-svrg-magma-occ}. The top row contains the original occluded image and reconstructed images by each algorithm and the middle row contains reconstructed noises as determined by each corresponding algorithm. FISTA (Figure \ref{p:fista-img-occ}) and MAGMA (Figure \ref{p:magma-img-occ}) both correctly reconstruct the underlying image while putting noise (illumination changes and gross corruptions) into the error part (Figures \ref{p:fista-err} and \ref{p:magma-err}). APCG (Figure \ref{p:apcg-img-occ}), on the other hand, reconstructs the true image together with the noise into the error part (Figure \ref{p:apcg-err}), as if the sought person does not have images in the database. SVRG seems to find the correct image in the database, however it fails to separate noise from the true image (Figures \ref{p:svrg-img-occ} and \ref{p:svrg-err}). We also report the reconstruction vectors $\mathbf{x}^\star$ as return by each algorithm in the bottom row. FISTA (Figure \ref{p:fista-occ}) and MAGMA (\ref{p:magma-occ}) both are fairly sparse with a clear spark indicating to the correct image in the database. SVRG (Figure \ref{p:svrg-occ}) also has a minor spark indicating to the correct person in the database, however it is not sparse, since it could not separate the noise. The result from APCG (Figure \ref{p:apcg-occ}) is the worst, since it is not sparse and indicates to multiple different images in the database, thus failing in the face recognition task.}

{Note that this is not a specifically designed case, in fact, SVRG, APCG and other algorithms that use partial information per iteration might suffer from this problem. Indeed, APCG uses one or a few columns of $\mathbf{A}$ at a time, thus effectively transforming the original large database into many tiny ones, which cannot result in good face recognition. SVRG, on the other hand, uses one row of $\mathbf{A}$ and all of variable $\mathbf{x}$ at a time for all iterations, which allows it to solve the recognition task correctly. However, when the dense error correction optimisation problem in (\ref{eq:L1regLS-CAB}) is used the minibatch approach results in using only one row of the identity matrix with each corresponding row of $\mathbf{A}$ and one coordinate from variable $\mathbf{e}$ together with $\mathbf{x}$, therefore resulting in poor performance when there is noise present.}

\begin{figure}
	\centering
	\begin{subfigure}[b]{0.19\textwidth}
	    \centering
		\includegraphics[scale=0.25]{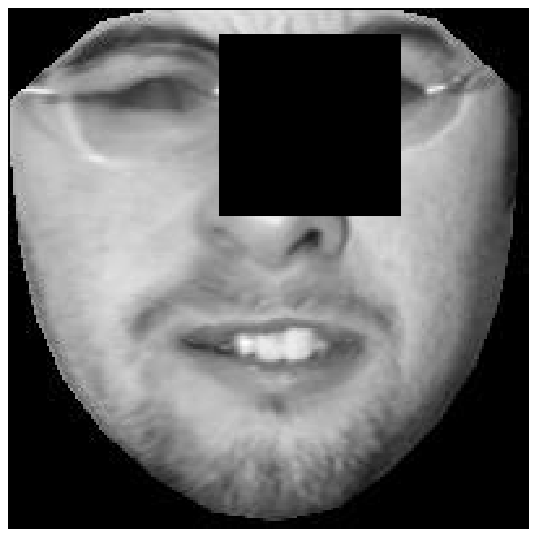}
		\caption{Original}
		\label{p:orig-occ}
	\end{subfigure}
	\begin{subfigure}[b]{0.19\textwidth}
	    \centering
		\includegraphics[scale=0.25]{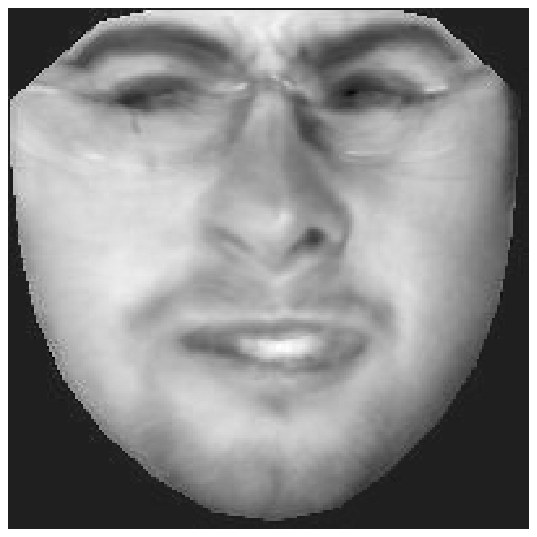}
		\caption{FISTA}
		\label{p:fista-img-occ}
	\end{subfigure}
	\begin{subfigure}[b]{0.19\textwidth}
	    \centering
		\includegraphics[scale=0.25]{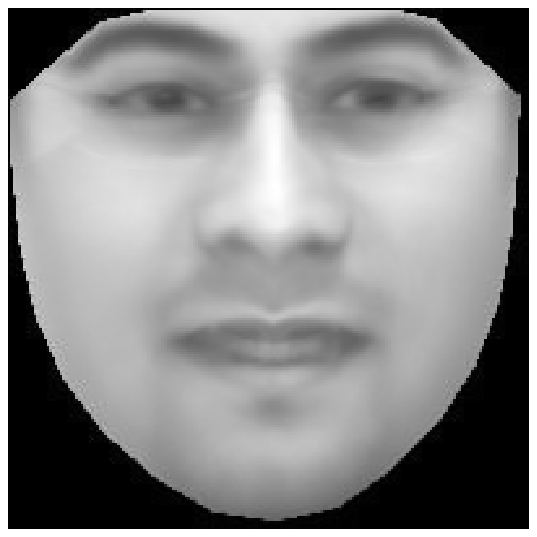}
		\caption{APCG}
		\label{p:apcg-img-occ}
	\end{subfigure}
	\begin{subfigure}[b]{0.19\textwidth}
	    \centering
		\includegraphics[scale=0.25]{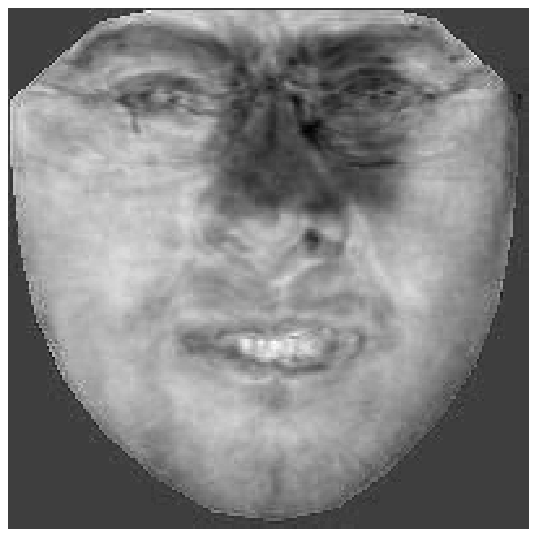}
		\caption{SVRG}
		\label{p:svrg-img-occ}
	\end{subfigure}
	\begin{subfigure}[b]{0.19\textwidth}
	    \centering
		\includegraphics[scale=0.25]{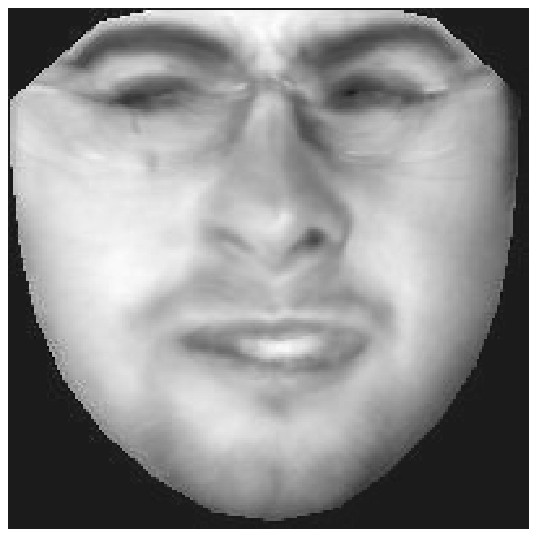}
		\caption{MAGMA}
		\label{p:magma-img-occ}
	\end{subfigure}
	
	\begin{subfigure}[b]{0.2\textwidth}
	    \centering
		\includegraphics[scale=0.25]{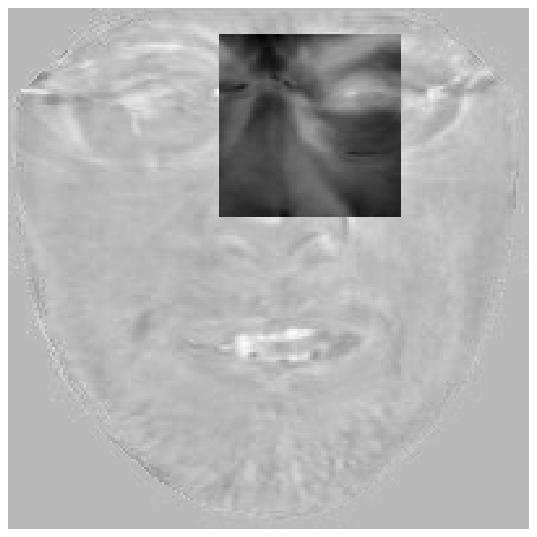}
		\caption{FISTA}
		\label{p:fista-err}
	\end{subfigure}
	\begin{subfigure}[b]{0.2\textwidth}
	    \centering
		\includegraphics[scale=0.25]{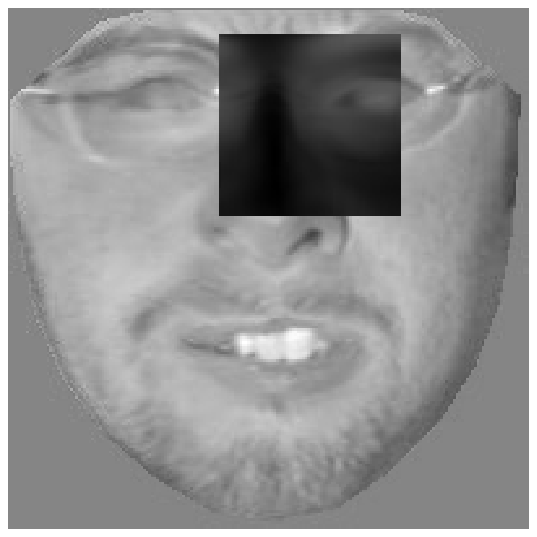}
		\caption{APCG}
		\label{p:apcg-err}
	\end{subfigure}
	\begin{subfigure}[b]{0.2\textwidth}
	    \centering
		\includegraphics[scale=0.25]{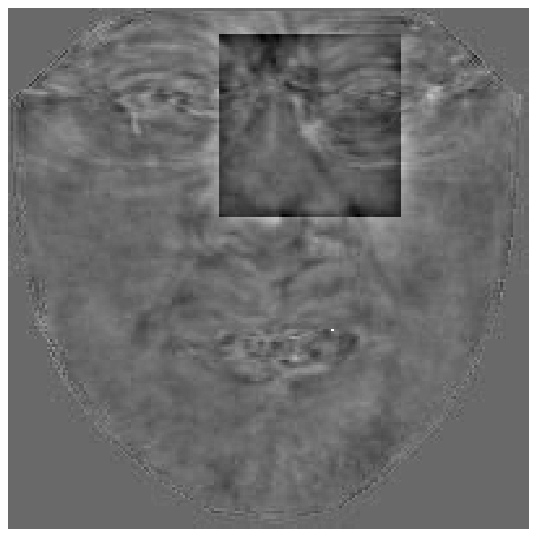}
		\caption{SVRG}
		\label{p:svrg-err}
	\end{subfigure}
	\begin{subfigure}[b]{0.2\textwidth}
	    \centering
		\includegraphics[scale=0.25]{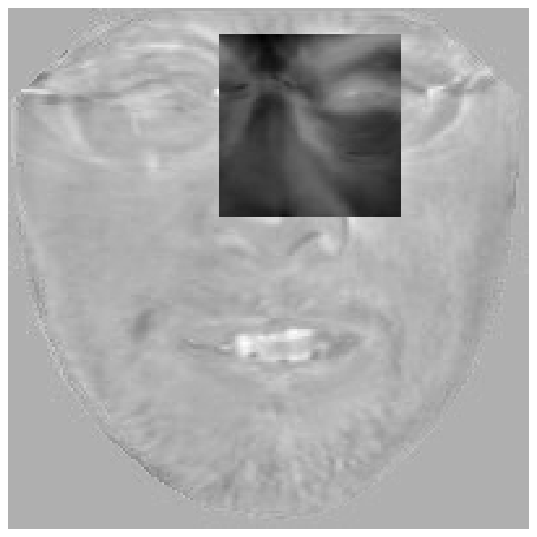}
		\caption{MAGMA}
		\label{p:magma-err}
	\end{subfigure}
	
	\begin{subfigure}[b]{0.24\textwidth}
	    \centering
		\includegraphics[scale=0.25]{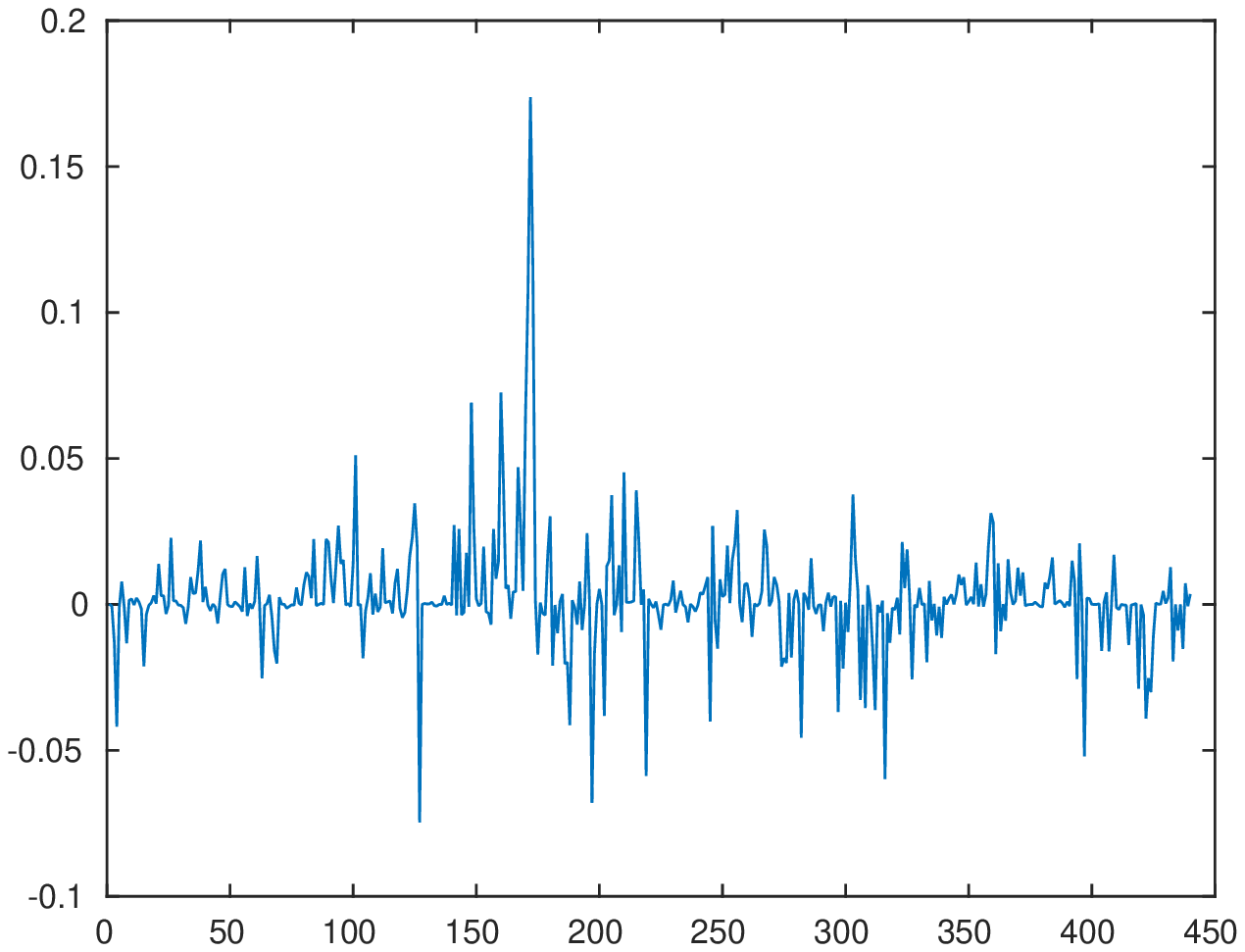}
		\caption{FISTA}
		\label{p:fista-occ}
	\end{subfigure}
	\begin{subfigure}[b]{0.24\textwidth}
	    \centering
		\includegraphics[scale=0.25]{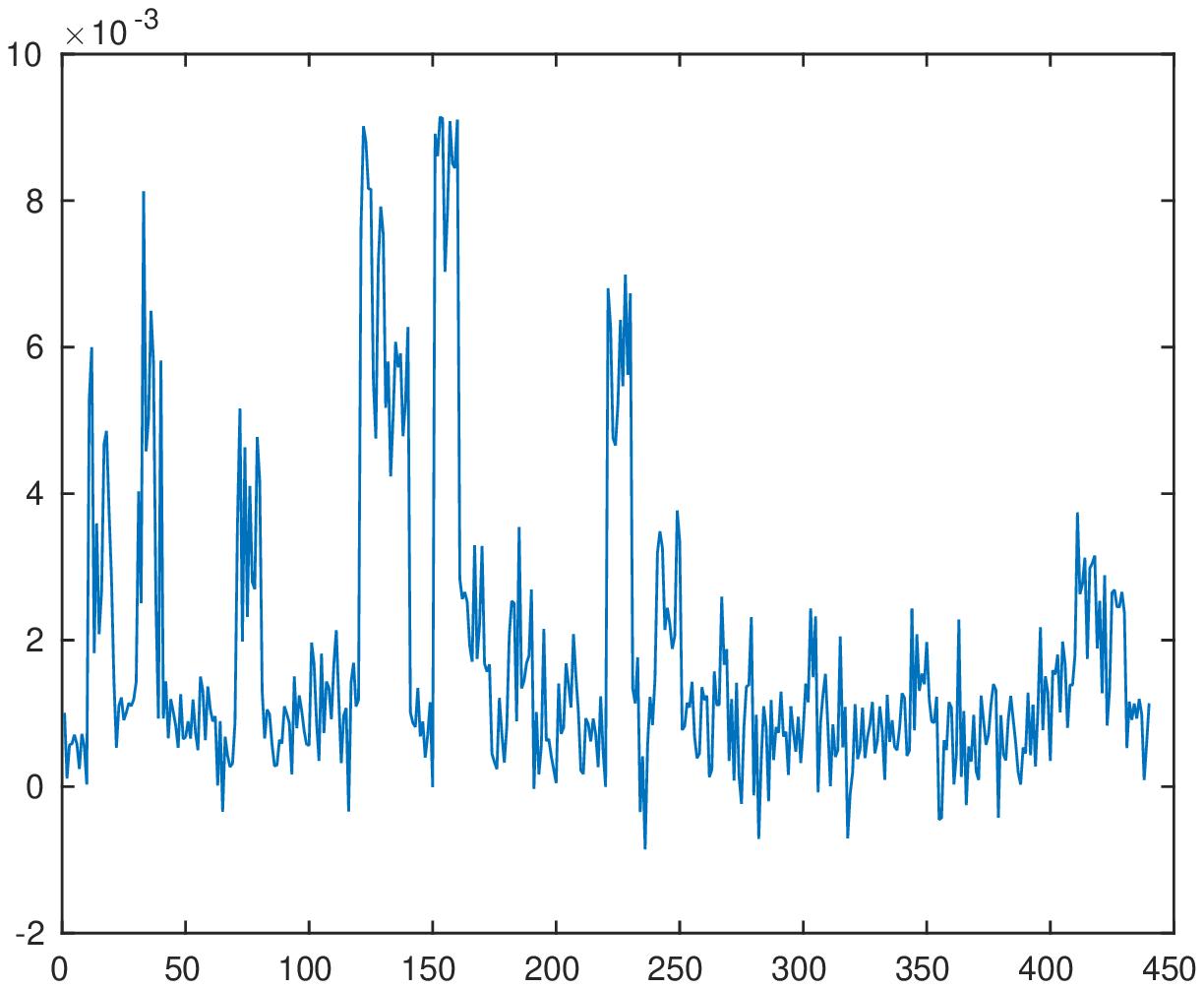}
		\caption{APCG}
		\label{p:apcg-occ}
	\end{subfigure}
	\begin{subfigure}[b]{0.24\textwidth}
	    \centering
		\includegraphics[scale=0.25]{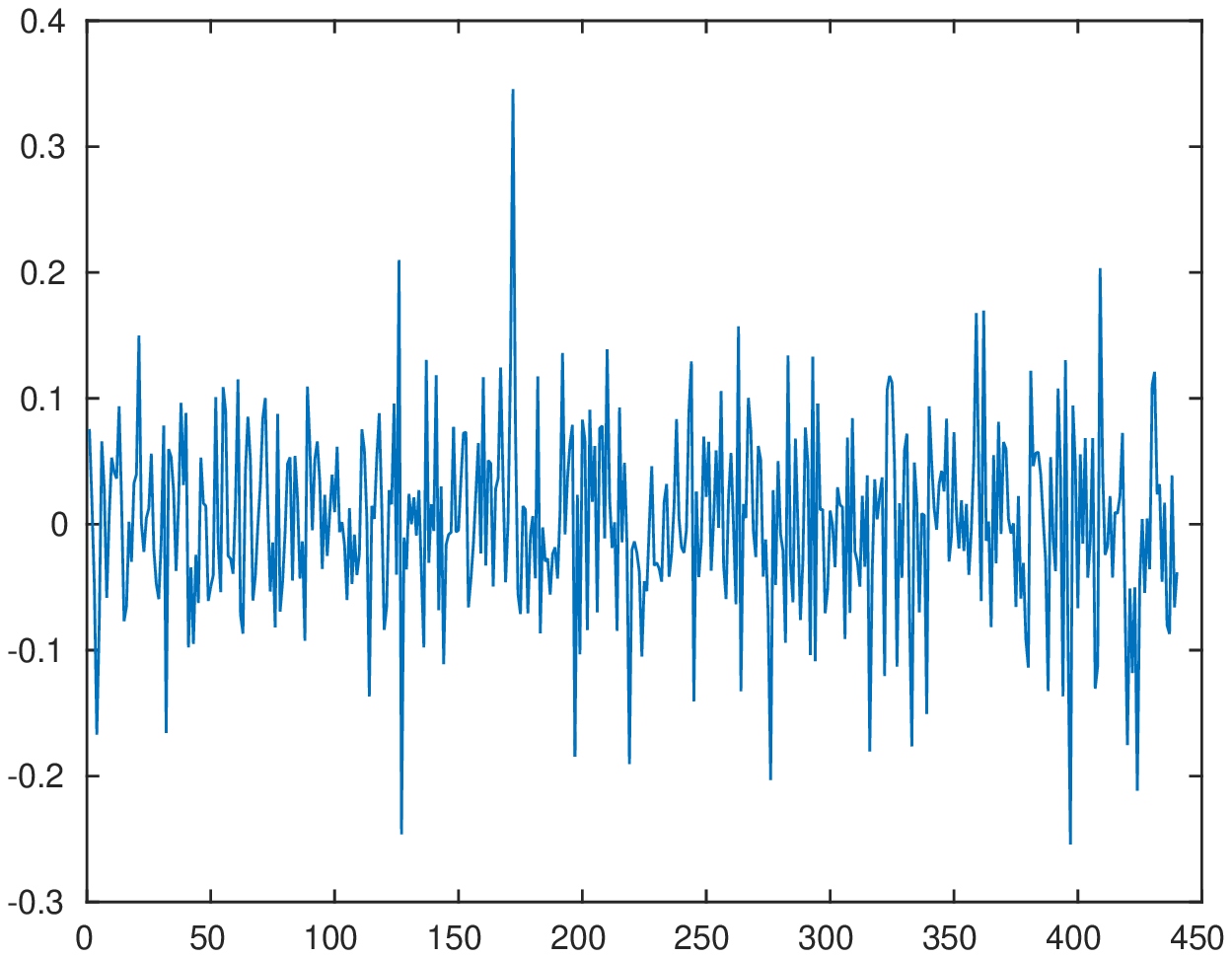}
		\caption{SVRG}
		\label{p:svrg-occ}
	\end{subfigure}
	\begin{subfigure}[b]{0.24\textwidth}
	    \centering
		\includegraphics[scale=0.25]{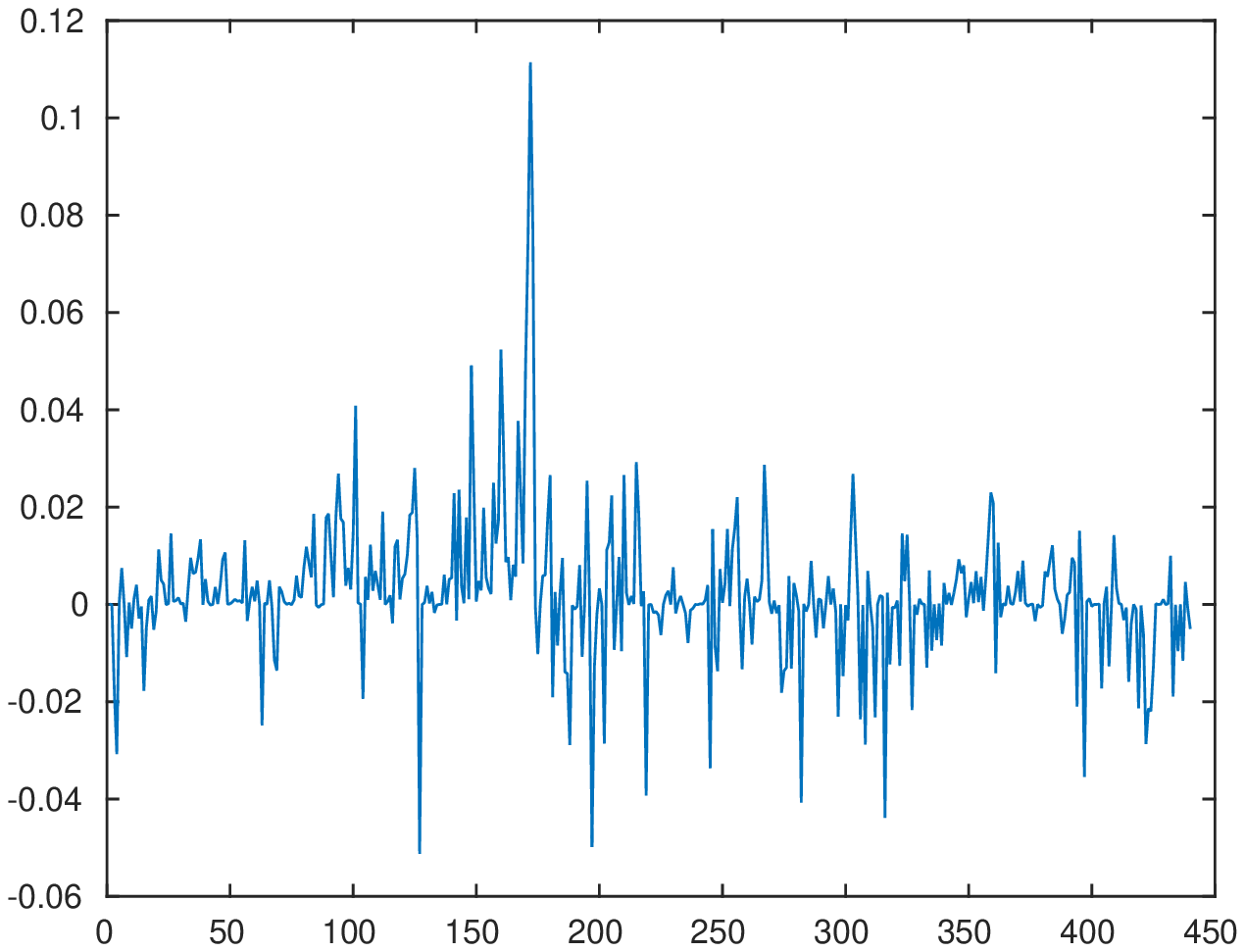}
		\caption{MAGMA}
		\label{p:magma-occ}
	\end{subfigure}
	\caption{{Result of FISTA, APCG, SVRG and MAGMA in an image decomposition example with corruption. Top row ((a)-(e)) contains the original occluded image and reconstructed images from each algorithm. The middle row ((f)-(i)) contains the reconstructed noise ($\mathbf{e}^\star$ reshaped as an image) as determined by corresponding algorithms. The bottom row ((j)-(m)) contains the solutions $\mathbf{x}^\star$ from each corresponding algorithm.}}
	\label{f:fista-svrg-magma-occ}
\end{figure}

{To further investigate the convergence properties of the discussed algorithms, we plot the objective function values achieved by APCG, {SVRG,} FISTA and MAGMA on problems with $440$, $944$ and $8824$ images in the database in Figure \ref{f:apcg_fista_magma} \footnote{For all experiments we use the same standard parameters.}. As we can see from Figures \ref{p:dic1} and \ref{p:dic2} APCG is slower than FISTA for smaller problems and slightly outperforms it only for the largest dictionary (Figure \ref{p:dic3}). SVRG on the other hand quickly drops in function after one-two iterations, but then stagnates. MAGMA is significantly faster than all three algorithms in all experiments.}

{However, looking at the objective function value alone can be deceiving. In order to obtain a full picture of the performance of the algorithms for the face recognition task, one has to look at the optimality conditions and the sparsity of the obtained solution. In order to test it, we run FISTA, APCG, {SVRG} and MAGMA until the first order condition is satisfied with $\epsilon=10^{-6}$ error. As the plots in Figure \ref{f:apcg_fista_magma-crit} indicate APCG and {SVRG} are slower to achieve low convergence criteria. Furthermore, as Table \ref{t:converge-times} shows they also produce denser solutions with higher $\ell_1$-norm.}

\begin{figure}
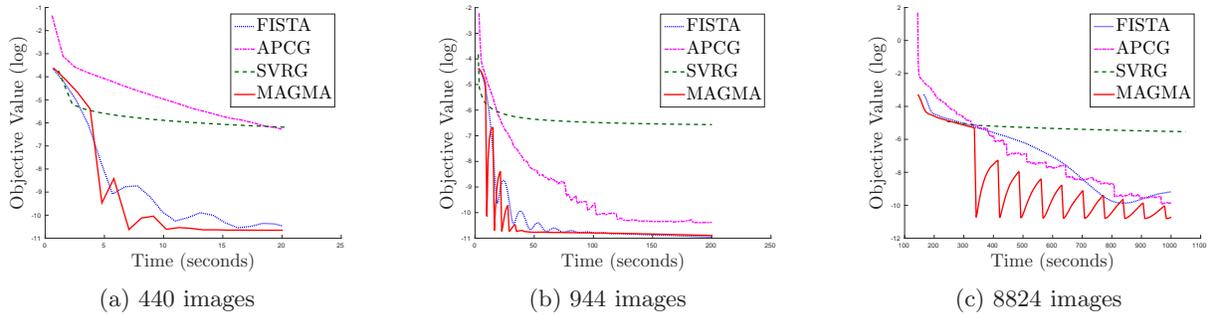

	\centering
	\begin{subfigure}[b]{0.32\textwidth}
	    \centering
		\includegraphics[scale=0.25]{dic1-val}
		\caption{$440$ images}
		\label{p:dic1}
	\end{subfigure}
	\begin{subfigure}[b]{0.32\textwidth}
	    \centering
		\includegraphics[scale=0.25]{dic2-val}
		\caption{$944$ images}
		\label{p:dic2}
	\end{subfigure}
	\begin{subfigure}[b]{0.32\textwidth}
	    \centering
		\includegraphics[scale=0.25]{dic3-val}
		\caption{$8824$ images}
		\label{p:dic3}
	\end{subfigure}
	\caption{{Comparing the objective values of FISTA, APCG, {SVRG} and MAGMA after running them for a fixed amount of time on dictionaries of various sizes.}}
	\label{f:apcg_fista_magma}
\end{figure}

\begin{figure}
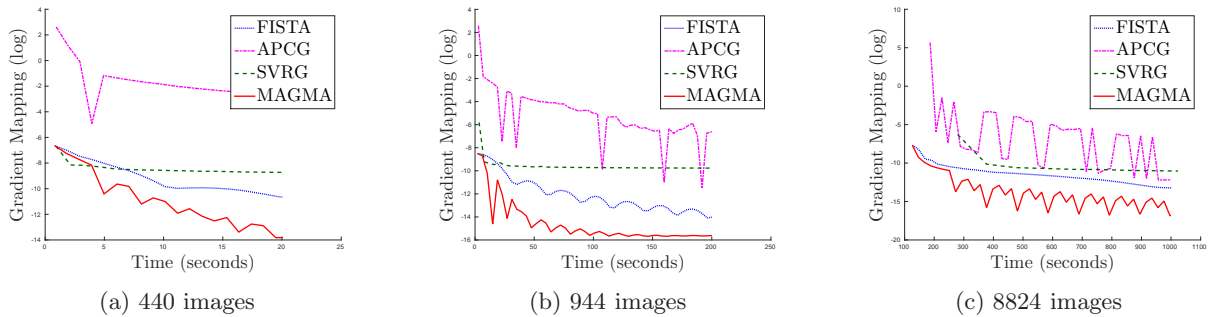

	\centering
	\begin{subfigure}[b]{0.32\textwidth}
	    \centering
		\includegraphics[scale=0.25]{dic1-crit}
		\caption{$440$ images}
		\label{p:dic1-crit}
	\end{subfigure}
	\begin{subfigure}[b]{0.32\textwidth}
	    \centering
		\includegraphics[scale=0.25]{dic2-crit}
		\caption{$944$ images}
		\label{p:dic2-crit}
	\end{subfigure}
	\begin{subfigure}[b]{0.32\textwidth}
	    \centering
		\includegraphics[scale=0.25]{dic3-crit}
		\caption{$8824$ images}
		\label{p:dic3-crit}
	\end{subfigure}
	\caption{{Comparing the first order optimality criteria of FISTA, {SVRG}, APCG and MAGMA after running them for a fixed amount of time on dictionaries of various sizes.}}
	\label{f:apcg_fista_magma-crit}
\end{figure}

\begin{table}
\centering
\begin{tabular}{|c|c|c|c|}
\hline 
Method & CPU Time (seconds) & Number of iterations & $\Vert\mathbf{x}^\star \Vert_1$ \\ 
\hline
FISTA & $1690$ & $554$ & $3.98$ \\
\hline
APCG & $680$ & $68,363$ (block size $10$) & $6.67$ \\
\hline
{SVRG} & $\geq 10,000$ & $\geq 93$ (epoch length $17,648$) & $5.19$ \\
\hline
MAGMA & $526$ & $144$ fine and $120$ coarse & $3.09$ \\
\hline
\end{tabular}
\caption{{Running FISTA, APCG, {SVRG} and MAGMA on the largest dictionary (with $8824$ images) until convergence with $\epsilon=10^{-6}$ accuracy or maximum time of $10,000$ CPU seconds returning $\mathbf{x}^*$ as a solution.}}
\label{t:converge-times}
\end{table}

\section{Conclusion and Discussion}

In this work we presented a novel accelerated multi-level algorithm - MAGMA, for solving convex composite problems. We showed that in theory our algorithm has an optimal convergence rate of $\mathcal{O}(1/\sqrt{\epsilon})$, where $\epsilon$ is the accuracy. To the best of our knowledge this is the first multi-level algorithm with optimal convergence rate. Furthermore, we demonstrated on several large-scale face recognition problems that in practice MAGMA can be up to $10$ times faster than the state of the art methods.

The promising results shown here are encouraging, and justify the use of MAGMA in other applications. MAGMA can be used to solve composite problems with two non-smooth parts. Another approach for applying FISTA on a problem with two non-smooth parts was given in \cite{orabona2012prisma}. This problem setting is particularly attractive, because of its numerous applications, including Robust PCA \cite{candes2011robust}. MAGMA's extension for solving computer vision applications of Robust PCA (such as image alignment \cite{peng2012rasl}) is an interesting direction for future research.

In terms of theory, we note that even though we prove that MAGMA has the same worst case asymptotic convergence rate as Nesterov's accelerated methods, it would be interesting to see, what conditions (we expect it to be high correlation of the columns of $\mathbf{A}$) imposed on the problem ensure that MAGMA has a better convergence rate or at least a strictly better constant factor. This could also suggest an automatic way for choosing parameters $\kappa$ and $K_d$ optimally and a closed form solution for the step size $s_k$.

\section*{Acknowledgements}
The authors would like to express sincere appreciation to the two anonymous referees for their useful suggestions and for having drawn the authors' attention to additional relevant literature. The first author's work was partially supported by \href{http://www.luys.am/}{Luys} foundation. The work of the second author was partially supported by EPSRC grants EP/M028240, EP/K040723 and an FP7 Marie Curie Career Integration Grant (PCIG11-GA-2012-321698 SOC-MP-ES). The work of S. Zafeiriou was partially funded by EPSRC project FACER2VM (EP/N007743/1).

\bibliographystyle{plain}
\bibliography{magma}

\end{document}